\documentclass[10pt, twoside, compsoc]{IEEEtran}
\ifCLASSINFOpdf
\else
\fi
\usepackage[brazil, english]{babel}
\usepackage{array}
\usepackage{amsmath, amssymb, amsthm, bm, amsfonts}
\usepackage[linesnumbered, ruled]{algorithm2e}
\usepackage{caption}
\usepackage{float} 
\usepackage{subcaption}
\usepackage{textcomp}
\usepackage{stfloats}
\usepackage{verbatim}
\newtheorem{assumption}{Assumption}
\usepackage{graphicx}
\newcommand{\myspace}[1]{\par\vspace{#1\baselineskip}}
\usepackage{indentfirst}
\usepackage[brazil, english]{babel}
\usepackage[colorlinks, 
            linkcolor=blue, 
            anchorcolor=red, 
            urlcolor=blue, 
            citecolor=green
            ]{hyperref}
\usepackage{pifont}
\usepackage{makecell}
\newtheorem{corollary}{Corollary}
\RequirePackage{CJKnumb}
\DeclareMathSizes{10}{10}{7}{5}
\usepackage{color, url, balance, times, helvet, courier}
\usepackage[square, comma, sort&compress, numbers]{natbib}
\usepackage{booktabs, threeparttable, multirow, adjustbox, threeparttable}
\newtheorem{proposition}{Proposition}[section]
\newtheorem{definition}{Definition}[section]
\newtheorem{Property}{Property}[section]
\newtheorem{theorem}{Theorem}[subsection]
\newtheorem{lemma}{Lemma}[subsection]
\theoremstyle{remark}
\newtheorem{remark}{Remark}
\usepackage[commandnameprefix=always]{changes}
\colorlet{Changes@Color}{red}
\usepackage{orcidlink}
\begin{document}
%\title{An Accelerated Adaptive Block Proximal Linear Framework for Nonconvex and Nonsmooth Optimization}
%\title{ABPL+: An Accelerated Block Proximal Linear Framework with Adaptive Momentum for Nonconvex Nonsmooth Optimization}
\title{An Accelerated Block Proximal Framework with Adaptive Momentum for Nonconvex and Nonsmooth Optimization}

\author{Weifeng Yang\orcidlink{0009-0004-9453-3525} and Wenwen Min$^*$\orcidlink{0000-0002-2558-2911} 
\IEEEcompsocitemizethanks{
\IEEEcompsocthanksitem 
Weifeng Yang and Wenwen Min contributed equally to this work, and they are with the School of Information Science and Engineering, Yunnan University, Kunming 650091, Yunnan, China. 
E-mail: minwenwen@ynu.edu.cn. 
}
\thanks{Manuscript received XX, 2023; revised XX, 2023. \newline
        (Corresponding authors: Wenwen Min)}}

\markboth{IEEE TRANSACTIONS ON KNOWLEDGE AND DATA ENGINEERING, 2023}{Yang \MakeLowercase{\textit{et al.}}:  Accelerated Adaptive Block Proximal Linear Method}

\IEEEtitleabstractindextext{
\begin{abstract}
We propose an accelerated block proximal linear framework with adaptive momentum (ABPL$^+$) for nonconvex and nonsmooth optimization. 
We analyze the potential causes of the extrapolation step failing in some algorithms, and resolve this issue by enhancing the comparison process that evaluates the trade-off between the proximal gradient step and the linear extrapolation step in our algorithm. Furthermore, we extends our algorithm to any scenario involving updating block variables with positive integers, allowing each cycle to randomly shuffle the update order of the variable blocks. Additionally, under mild assumptions, we prove that ABPL$^+$ can monotonically decrease the function value without strictly restricting the extrapolation parameters and step size, demonstrates the viability and effectiveness of updating these blocks in a random order, and we also more obviously and intuitively demonstrate that the derivative set of the sequence generated by our algorithm is a critical point set. Moreover, we demonstrate the global convergence as well as the linear and sublinear convergence rates of our algorithm by utilizing the Kurdyka–Łojasiewicz (KŁ) condition. To enhance the effectiveness and flexibility of our algorithm, we also expand the study to the imprecise version of our algorithm and construct an adaptive extrapolation parameter strategy, which improving its overall performance. We apply our algorithm to multiple non-negative matrix factorization with the $\ell_0$ norm, nonnegative tensor decomposition with the $\ell_0$ norm, and perform extensive numerical experiments to validate its effectiveness and efficiency. 
\end{abstract}

}

\maketitle\IEEEpeerreviewmaketitle

\section{Introduction}
\IEEEPARstart{I}n many machine learning applications, a broad class of nonconvex and nonsmooth problems can be presented as: 
\begin{small}
\begin{align}\label{e11}
    \min\limits_{(\left\{{x_{i}} \right\}^{n}_{i=1}\ \in \prod_{i}^n \mathbb{R}^{d_{i}})}J(\left\{{x_{i}} \right\}^{n}_{i=1})=H(\left\{{x_{i}} \right\}^{n}_{i=1})+\sum_{i=1}^{n}F_{i}(x_{i}),
\end{align}
\end{small}
where $d_{i} \in \mathbb{N}, ~F_{i}: \mathbb{R}^{d_{i}} \rightarrow \mathbb{R}, ~H: \prod_{i}^n \mathbb{R}^{d_{i}} \rightarrow\mathbb{R}$. $F_{i}$ is a proper, lower semicontinuous (possibly nonconvex and nonsmooth) function, $H$ are continuously differentiable (possibly nonconvex) functions (see Assumption \ref{assump1} in the following section for a clearer definition). Obviously, it can be expressed as the following $n + 1$ compound function problems which have been seen in a wide array of problems, including:  Matrix Factorization  \cite{lee1999learning,bolte2014proximal,min2021Group}, Logistic Regression  \cite{wang2014clinical,min2018network}, Tensor Decomposition  \cite{cong2015tensor,min2021tscca}, Image Classification  \cite{zhang2018classify}, Cardinality Constrained Portfolio \cite{shi2022cardinality}, Generalized Iterated Shrinkage Thresholdings \cite{zhang2020global}, etc. 

The problem of Eq. (\ref{e11}) usually does not use the method of joint optimization, but instead it is usually the method of alternate optimization, that is block coordinate descent (BCD) method: 
\begin{align}
    x^{k+1}_{j} \in \min\limits_{x_{i} \in R^{d_{i}}} J(\left\{{x^{k+1}_{i}}\right\}^{j-1}_{i=1}, x_{j}, \left\{{x^{k}_{i}} \right\}^{n}_{i=j+1}). 
\label{e12}
\end{align}

But BCD methods has a defect that it need to be able to obtain an analytical solution \cite{boyd2004convex} before proceeding to the next step, the cost is extremely high for solving problem. such as sparse non-negative matrix factorization with $\ell_0$-norm  \cite{bolte2014proximal}, which is an NP-hard and non-convex problem.

To solve this problem, an efficient method based on the proximal operator is proposed, which works by linearizing the smooth term in the subproblem, it is called the proximal alternate linearization minimization (PALM) algorithm  \cite{bolte2014proximal}, when $H$ is continuously differentiable: 
\begin{align}
    && x^{k+1}_{j} &\in \min\limits_{x_{j} \in R^{d_{i}}} \Big[F_{j}(x)+\langle  x_{j}, \nabla_{x_{j}} H(\left\{{x_{i}} \right\}^{n}_{i=1}) \rangle \notag \\
    && &+\frac{1}{2\sigma_{j}} ||x_{j}-x^{k}_{j}||^{2} \Big],
\label{e14}
\end{align}
where $\sigma_{i}$ is local Lipschitz constant of $\nabla_{x_{j}} H(\left\{{x_{i}} \right\}^{n}_{i=1})$. In this way, the solution of each sub-problem can be solved more easily. And the PALM algorithm has been applied to many problems, such as Dictionary Learning for Sparse Coding \cite{bao2015dictionary}, Sparse Matrix Factorization \cite{min2022structured}. 

The PALM algorithm framework needs to satisfy the following three main properties  \cite{bolte2014proximal}: 
\begin{Property}   
(\romannumeral1) Sufficient decrease property, i.e., there exists a positive constant $\rho$ such that $\rho||x^{k+1}-y^{k}||\leq \phi(x^{k})-\phi(x^{k+1})$.

(\romannumeral2) A subgradient lower bound for the iterates gap: i.e., 
there exists a positive constant $\rho_{2}$ such that $\rho_{2}||x^{k+1}-y^{k}||\geq ||g^{k}||$ where $g^{k} \in \partial \phi(x^{k})$.

(\romannumeral3) Using the Kurdyka–Łojasiewicz (KŁ) Property (see more precise definitions in the next section and \cite{bolte2014proximal}):  Assume that $\phi$ is a KŁ function and show that the generated sequence $\left\{x^{k}\right\}(k\in \mathbb{N})$ is a Cauchy sequence. 
\label{based}
\end{Property}  
It is worth noting that the term $\phi$ used above does not always correspond to the objective function. In fact, $\phi$ can be an auxiliary function of $J(x^{k})$ and $||x^{k}-x^{k-1}||$, such as iPALM \cite{pock2016inertial}. 

% It is obvious that in order to employ any approach associated with proximal gradient methods  \cite{attouch2010proximal}, the first two requirements must be satisfied. It is possible to examine whether the third requirement is true under the presumption that the objective function satisfies the KŁ Property, along with the first two main properties. 
Notably, it has been demonstrated that the proximal gradient algorithm and proximal alternating linear minimization algorithm satisfies Property \ref{based} without acceleration \cite{attouch2013convergence}. 

Therefore, Accelerated Proximal Alternative Linear Minimization (APALM) algorithms, which draw inspiration from the accelerated proximal gradient approach, develop naturally. In order to ensure that the accelerated alternate linear minimization algorithm satisfies Property \ref{based}, a straightforward approach is to restrict the range of extrapolation parameters and step size. Notable examples of the APALM algorithms that satisfy Property \ref{based} by limiting the selection range of extrapolation parameters and step size include iPALM  \cite{pock2016inertial} and BPL  \cite{xu2017globally}. The former advocates the use of auxiliary functions, such as the construction of $\phi(x^{k}, x^{k-1})=J(x^{k})+\frac{\delta}{2}||x^{k} - x^{k-1}||$ to disguised satisfaction of the monotonically decreasing property in Property \ref{based}, however, this comes at a cost because the objective function will not monotonically decrease even if the auxiliary function satisfies the Property \ref{based}, which greatly reduces the algorithm's numerical performance, moreover, since the selection range of the extrapolation parameters and step size must be rigorously constrained because they are not independent of one another, the extrapolation effect will be further diminished. And a number of better methods for the iPALM algorithm have been put forth. For instance, GiPALM  \cite{gao2020gauss} meets Property \ref{based} by altering the iPALM algorithm's extrapolation step order, whereas NiPALM  \cite{wang2023generalized} satisfies Property \ref{based} by lengthening the auxiliary function sequence. IBPG and IBP \cite{le2020inertial} utilize complicated and stringent selection criteria and stronger assumptions in order to expand the extrapolation parameters' selection range and to meet Property \ref{based}, for example, IBPG and IBP call for a strong convexity in the coupling function, which will render some special problems, like Logistic Regression \cite{wang2014clinical}, invalid. 

The latter approach advocates a fundamentalist perspective, that is, explicitly limiting the selection range of extrapolated parameters and step size, and employing the backtracking method to determine if extrapolated values are feasible, so as to ensure the goal function monotonically decreases, this type of algorithm is represented by BPL  \cite{xu2017globally}. These algorithms concentrate on carefully choosing extrapolation parameters and step size to preserve Property \ref{based}. But the backtracking method's numerical effectiveness is not very great, and occasionally it is far less effective than simply setting the restart step. Furthermore, the method's extrapolation step size and the parameter selection range are not independent of one another. Examples of such algorithms include: a specific algorithm IR-t-TNN \cite{wang2021generalized} for tensor recovery problems, etc. 

These kind of algorithm that limits the range of extrapolation parameters and step size will obviously reduce the acceleration effect, and this kind of algorithm also has the defect that it cannot adaptively modify the extrapolation parameters, which further reduces the acceleration effect. As the above saying describes, Property \ref{based} is completely satisfied by the non-accelerated proximal gradient technique. As a result, an algorithm can also satisfy Property \ref{based} by inserting a restart step that contrasts the linear extrapolation step with the proximal gradient step. APGnc \cite{yao2016more} is an example of this kind of algorithm. But this kind of approach has intrinsic limitations that the extrapolation step fails in some special problem (In Section \ref{section: analysis} and Appendix \ref{Appdix2}, we will provide a thorough analysis of it), despite the fact that it can be adaptively altered and does not impose rigid limits on the extrapolation parameters and step size. It is also important to note that most APGnc algorithms are derived from accelerated proximal gradient algorithms(APG) \cite{beck2009fast} rather than proximal alternating linear minimization algorithms (PALM), such as mAPG \cite{li2015accelerated} and APGnc$^{+}$ \cite{li2017convergence}, so these algorithms are typically restricted to situations that univariate optimization problems and rarely determine the algorithm's global convergence using the KŁ property \cite{li2015accelerated,yao2016more,li2017convergence}. Thus, to solve these problems, a new extended algorithm framework and a new analysis method are needed. 

\subsection{Contributions}
(1) We propose an accelerated proximal alternating linear minimization method with adaptive momentum. Our algorithm not only has the properties of a monotonically decreasing objective function and adaptive momentum step size, but also effectively addresses the issue of strict constraints on extrapolation parameters and step size in general accelerated proximal alternating minimization algorithms such as BPL \cite{xu2017globally} and iPALM \cite{pock2016inertial}, and solves the inherent flaw observed that the extrapolated step fails and the inapplicability of multivariate optimization problems in the APGnc \cite{li2017convergence} algorithm framework, while showing the effectiveness of our algorithm under random sorting order. 

(2) We show the reasons for the failure of some algorithm framework extrapolation steps and solve it with our algorithm. Under milder assumption, we not only ensures the monotonically decreasing objective function of our algorithm, but also demonstrates the convergence and efficiency of our algorithm's update order randomization technique. Additionally, we more intuitively and succinctly demonstrates that the sequence produced by the algorithm converges to a critical point. 
% and acquires some by-product properties.
For objective functions satisfying the Kurdyka–Łojasiewicz Property, we also obtain the global convergence and convergence rate for our algorithm.

(3) In Section \ref{section: B}, our algorithms have been applied to the sparse non-negative tensor decomposition with $\ell_0$-norm constraints and multiple sparse non-negative matrix factorization with $\ell_0$-norm constraints. Numerical experiments validate the effectiveness of our proposed algorithm. 

\section{Symbol definitions and preliminaries}
\label{section: pre}
In this section, we will introduce some symbol definitions and the preliminary knowledge required in the derivation process. Table \ref{notation} summarizes the notation used in the paper. 

\begin{table}[!ht] 
\renewcommand\arraystretch{1.5}
\centering % 表格整体居中
\caption{Notation}
\begin{tabular}{|p{3cm}|p{5cm}|} \hline% 其中, |c|表示文本居中, 文本两边有竖直表线。
Notation & Definition \\ \hline
$\left\{{x_{i}} \right\}^{n}_{i=1}$ & $(x_{1}, x_{2}. . . . . . , x_{n})$ \\
$\left[ n \right]$  & $\left\{i\right\}_{i=1}^{n}$ \\
$\left\{{x_{i}} \right\}_{i\in T_{s}}$ or $(x_{T_{s}})$  & $(x_{T_{s_{1}}}, x_{T_{s_{2}}}. . . . . . , x_{T_{s_{card(T_{s})}}})$ where $T_{s}$ is a random sort set of $\left[ n \right]$. \\
$card(S)$ & The number of elements in the $S$ set. \\

$x_{(n)}$  & $\left\{{x_{i}} \right\}^{n}_{i=1}$ \\ 
$x_{(n)}-y_{(n)}$ &	$\left\{{x_{i}}-{y_{i}}\right\}^{n}_{i=1}$ \\ 
$L_{j}$ and $L_{\nabla H_{x_{j}}}$ & the lipchitz constant denoting the variable of the $j$-th block of the function $H(x_{(n)})$ \\ 
$x^{k}_{i}$  & the $i$-th block of $x$ within the $k$-th outer loop\\ 
$\mathcal{C}^{p}_{L}(X)$ & a collection of functions satisfying the lipschitz property \\

$\langle \centerdot, \centerdot \rangle$ & inner product \\
$||\centerdot||_{p}$ & $l_{p}$ norm \\
$||\left\{{x_{i}}\right\}^{n}_{i=1}||$ & $ \sum_{i=1}^{n}||(x_{i})|| \ ({\forall} x_{i} \in R^{n})$\\ \hline
% $[[ \ ]]$ & Kruskal operator \\ \hline
\end{tabular}
\label{notation}
\end{table}

\subsection{Symbol definition in point set topology, measure theory and set theory}
In this paper, all topological spaces are usual metric topology. Thus, there are some definitions and properties \cite{ciarlet2013linear,armstrong2013basic,stein2009real,foreman2009handbook,halmos2013measure}. 
\begin{definition}
Set $(X,\tau) $ is a topological space, 
% $A \subseteq X$ is a open set if and only if $~{\forall} x \in A, ~{\exists} U \in \tau, ~ x \in U \subseteq A$,
$A \subseteq X$ is a closed set if and only if the complement of A is open or A is equal to its closure $\bar{A}$. 
\label{dopen}
\end{definition}
\begin{definition}
If A is a set, then defined the $A'$ is set of overall cluster points of A, $A^{o}$ is set of overall inter points of A. 
% $A^{c}$ is complementary set of A. 
\label{d1}
\end{definition}
\begin{proposition}
Set $(X,\tau) $ is a topological metric space, then ${\forall}A \subseteq X$, $A'$ is closed. 
\label{dtop}
\end{proposition}
\begin{proposition}
A bounded closed set in a finite-dimensional space is a compact set. 
\label{dcompact}
\label{p4}
\end{proposition}
\begin{proposition}
The closure of a bounded set is still bounded.
\label{dbouned}
\end{proposition}
\begin{definition}
Let M and N be any two sets. If there is a bijection $\Phi$: $M\rightarrow N$, then $|M|=|N|$ or $card(M)=card(N)$, and if there is an injective $\varphi$: $M\rightarrow N$, then $|M|\leq|N|$.
\label{dcard}
\end{definition}
\begin{proposition}
Zorn's lemma: let $X$ be a non-empty partially ordered set. If every totally ordered subset of $X$ has an upper bound in $X$, then $X$ has at least one maximal element.
\label{zorn}
\end{proposition}
\begin{proposition}
In usual metric topological space, the set of real numbers is a natural order.  
\label{natureorder}
\end{proposition}
\begin{definition}
For $S \subseteq \mathbb{R}^{n}$, we define the distance between the S and the point x as: $\operatorname{dist}(x, S): =\operatorname*{inf}\,\left\{\|x-y\|: \,y\in S\right\}. $
% \begin{center}
% $\operatorname{dist}(x, S): =\operatorname*{inf}\, \left\{\|x-y\|: \, y\in S\right\}. $
% \end{center}
\end{definition} 
\begin{definition}
Given a space X, a function $\mu$:  X $\rightarrow [0,\infty]$ is said to be a measure if it satisfies the following properties: \\
(\romannumeral1): For the empty set $\emptyset$, $\mu(\emptyset) = 0$. \\
(\romannumeral2): Countable additivity: for any countable collection of disjoint sets $\left\{A_{i}\right\}_{i=1}^{\infty} \subseteq X$: $\mu\left(\bigcup_{i=1}^{\infty} A_i\right) = \sum_{i=1}^{\infty} \mu(A_i)$.
\label{dmeasure}
\end{definition}
\begin{definition}
Lebesgue measure can be definited as: 
\begin{center}
$\mu(E) = \inf \left\{\sum_{i=1}^{\infty} \operatorname{vol}(I_{i}):  E \subseteq \bigcup_{i=1}^{\infty} I_{i} \right\}$,
\end{center}
where $E$ is the set to be measured, $\mu(E)$ represents its Lebesgue measure, $I_{i} \subseteq \mathbb{R}^{n}$, $\operatorname{vol}(I_{i})$ represents the volume of rectangle $I_{i}$, $\bigcup_{i=1}^{\infty} I_{i}$ represents the union of all rectangles, and $\inf$ represents the minimum of all possible rectangle covers. 
\label{lebesgue}
\end{definition}
\begin{definition}
From Definition \ref{dmeasure}, if $\mu(X)=1$, then $\mu$ is a probability.
\label{pro}
\end{definition}
\begin{definition}
Given ${\left\{k_{i}\right\}}_{i=1}^{m,m \leq n} \subseteq \left[ n \right]$, a subset $ A \subseteq X$, and $(x_{\left[ n \right]}) \in X$, then: 
\begin{center}
$A|_{{\left\{k_{i}\right\}}_{i=1}^{m,m \leq n}}=\left\{(x_{\left[ n \right]  }): (x_{{\left\{k_{i}\right\}}_{i=1}^{m,m \leq n}}) \in A \right\}$.
\end{center}
$A|_{{\left\{k_{i}\right\}}_{i=1}^{m,m \leq n}}$ is also called the truncated set of A. 
\label{truncated}
\end{definition}
\begin{lemma}
If both $E_{1}$ and $E_{2}$ are measurable sets, according to Definition \ref{dmeasure}, then: $\mu(E_{1} \times E_{2})=\mu(E_{1}) \mu(E_{2})$.
\label{cmeasure}
\end{lemma}
\subsection{Notation and preliminaries for nonconvex analysis}
\begin{definition}
Proper Function: A function $g: \mathbb{R}^{n} \rightarrow (-\infty,+\infty]$ is said to be proper if $\mathrm{dom}$ $g \neq \emptyset$, where $\mathrm{dom} ~g =\left\{x \in \mathbb{R}:~ g(x)<\infty \right\}$. 
\end{definition}
\begin{definition}
Lower Semicontinuous Function: if $\lim_{k\rightarrow \infty} x_{k}=x, ~ f(x) \leq \operatorname*{lim}\operatorname*{inf}_{k\to\infty}f(x_{k}) ({\forall} x_{k} \in \mathrm{dom}~ f)$, then f is called lower semicontinuous at $\mathrm{dom}~ f$. 
\end{definition}
\begin{definition}
Coercive Function: if f is coercive, then $\left\{x | x \in R^{n}, ~ f(x) < a, ~{\forall} a \in \mathbb{R} \right\}$ is bounded and $\inf_{x} f(x) > -\infty$.
\label{d0}
\end{definition}
\begin{definition}
Let f be a proper lower semicontinuous function, The Fréchet subdifferential of f at x, written ${\hat{\partial}} f(x)$, is the set of all vectors u which satisfy: 
\begin{center}
$\operatorname*{lim}_{y\neq x,y\to x}\cdot\frac{f(y)-f(x)-\langle u,\ y-x\rangle}{\|y-x\|}\ge0, $
\end{center}
when $ x \notin \mathrm{dom}~ f$, then set ${\hat{\partial}} f(x)=\emptyset$.
\label{d2}
\label{dsubdiff}
\end{definition}
\begin{definition}
The limiting subdifferential: $\partial f(x): =\{u\in\mathbb{R}^{n}: \exists x^{k}\to x,f(x^{k})\to f(x),u^{k}\to u,u^{k}\in\widehat{\partial}f(x^{k})\}.$
\end{definition}
% \begin{definition}
% \begin{align}
% \partial f(x): =\{u\in\mathbb{R}^{n}: \exists x^{k}\to x,f(x^{k})\to f(x),u^{k}\to u,u^{k}\in\widehat{\partial}f(x^{k})\}. \notag
% \end{align}
% \end{definition}
\begin{proposition}
Fermat’s lemma: Let f be a proper lower semicontinuous function. If f has a local minimum at $x^{*}$, then $0 \in \partial f(x^{*})$. 
\label{p1}
\end{proposition}
\begin{proposition} \cite{bolte2014proximal}
Let f be a proper lower semicontinuous function, and g be a continuously differential function. Then for any $x \in \mathrm{dom} ~f$, $\partial$(f+g)(x) = $\partial$f (x) +$\nabla$g(x). 
\label{p2}
\end{proposition}
\begin{lemma}
Heine's\ theorem: $E$ is domain of function f, then ${\forall}x_{n}, ~ x _{0} \in E, ~ x_{0} \neq x_{n}, ~ \lim_{n \rightarrow \infty} x_{n}=x_{0}$, then: 
\begin{center}
$\lim_{x\rightarrow x_{0}}f(x)=f(x_{0}) \Leftrightarrow \lim_{n \rightarrow \infty}f(x_{n})=f(x_{0})$,
\end{center}
\label{tf4}
\end{lemma}
% And next, the concept of definitions of some special functions in metric spaces will be given: 
\begin{definition}
Set\ f: $X\rightarrow R, ~ X \subseteq R^{n}$, define the set $\mathcal{C}^{p}_{L}(X)$ as a set composed of all functions satisfying the following properties: 
   \begin{equation}\label{e21}
    ||\nabla^{p} f(x)-\nabla^{p} f(y)|| \leq L*||x-y||. 
   \end{equation}
$\mathcal{C}^{p}_{L}(X)$ is also known as a collection of functions satisfying the lipschitz Property. in particular, 
% when p=0, f is lipschitz continuous, and 
when p=1, $\nabla_f$ is lipschitz continuous. 
% when f satisfies: $f\in \mathcal{S}_{\mu}(X) \cap \mathcal{C}^{p}_{L}(X)$, then $f \in \mathcal{S}^{p}_{\mu, L}(X)$. 
   \label{dlipchitz}
   \label{d3}
\end{definition} 
% If  $f \in \mathcal{C}^{1}_{L}(R^{n})$, which has a proposition: 
\begin{proposition}
Set\ f: $R^{n}\rightarrow R$, if  $f \in \mathcal{C}^{1}_{L}(R^{n})$, then:  
 $f(y) \leq f(x)+\langle \nabla f(x),y-x \rangle+\frac{L}{2}||y-x||^{2}.$
\label{p3}
\end{proposition}
\begin{definition}
Let f be a proper lower semicontinuous function, f is said to have the KŁ property on $\bar{u}\in d o m(\partial f) $, if there exists $\eta\in\left(0, +\infty\right]$, and U is the neighborhood of $\bar{u}$, Any u in U that satisfies the condition $f({\bar{u}}) < f(u) < f({\bar{u}})+\eta$ has the following inequalities established: 
\begin{center}
$\phi^{\prime}(f(u)-f(\bar{u}))d i s t(0,\partial f(u))> 1$,
\end{center}
where $\phi$ is the desingularization function, i.e.: $\phi \in C^{1}([0,\eta)), ~ \phi '>0, ~ \phi(0)=0$, and $\phi$ is a concave function from in $(0,\eta)$. If f has the KŁ property at each point of $\mathrm{dom}$ $\partial f$, then f is a KŁ function
\label{d8}
\end{definition}
% The inequality comes from Łojasiewicz \cite{Lojasiewicz1963propriete} and Kurdyka  \cite{kurdyka1998gradients}, which is extended in  \cite{bolte2007clarke} to the case of non-smooth functions. 
% \begin{lemma}
% (Uniformized KŁ Property  \cite{bolte2014proximal}): 
% Let $\Omega$ be a compact set and let $\sigma$ be a proper and lower semicontinuous function. Assume that $\sigma$ is constant on $\Omega$ and satisfies the KŁ Property at each point of $\Omega$. Then, there exist $\epsilon>0, \eta>0$ and $\phi \in \Phi_{n}$ such that for all $\overline{u} \in \Omega$  and all u in the following intersection: 
% \begin{center}
% $\left\{u\in\mathbb{R}^{d}: \, ~ d i s t\, ~ (u,\Omega )< \varepsilon\right\}\cap\left[\sigma\left({\overline{{u}}}\right) < \sigma\left(u\right) < \sigma\left({\overline{{u}}}\right)+\eta\right]$.
% \end{center}
%  one has: 
%  \begin{center}
% $\varphi^{\prime}\left(\sigma\left(u\right)-\sigma\left(\overline{{{u}}}\right)\right)d i s t\left(0,partial\sigma\left(u\right)\right)\ge1. $
% \end{center}
% \label{td1}
% \end{lemma}
Kurdyka–Łojasiewicz (KŁ) properties \cite{bolte2007clarke} play a very important role for global convergence analysis in the non-convex optimization. 

Following \cite{bolte2014proximal}, we take the following as our blanket assumption for problem (\ref{e11}): 
\begin{assumption}
\myspace{0}(i) Set $E_{H} \subseteq \prod_{i}^n \mathbb{R}^{d_{i}}$, $H: E_{H} \rightarrow\mathbb{R} $, $H \in \mathcal{C}^{1}_{L}(X)$ and H is continuously differentiable. 
\myspace{0}(ii) Set $E_{F_{i}} \subseteq \prod_{i}^n \mathbb{R}^{d_{i}}$, $F_{i}: E_{F_{i}} \rightarrow \mathbb{R} $ are proper, lower semicontinuous functions. 
\myspace{0}(iii) The objective function J satisfies the KŁ Property, and J is a proper coercive function, and J is lower bounded. 
\label{assump1}
\end{assumption}
\begin{remark}
Assumption \ref{assump1} (iii) means the sequences generated by algorithms are bounded.
\end{remark}

\section{Proposed APBL$^+$ framework}
\label{section: A}
\label{section: analysis}
In this section, we will analyze shortcomings of some algorithms, and propose our algorithm framework to solve these shortcomings.
At the same time, we will also demonstrate the convergence of the proposed algorithms in next section.
 
At first, let $S$ be the domain of function $J$, the proximal operator plays an important role in the following analysis. With the problem (\ref{e11}), we define proximal operator as: 
\begin{equation}
\begin{aligned}
&&x^{k+1}_{j} &\in \operatorname*{\arg\min} \limits_{x \in S^{d_{j}} \subseteq S} (F_{j}(x_{j})+\frac{1}{2\sigma_{j}}||x-y^{k}_{j}||^{2} \\
&& \ &+\langle \nabla_{x_{j}} H(\left\{{g_{i}} \right\}^{j-1}_{i=1}, y^{k}_{j}, \left\{{g_{i}} \right\}^{n}_{i=j+1}), x-y^{k}_{j} \rangle).  
\end{aligned}
\label{e32}
\end{equation}

As the introduction say, in order to solve problem (\ref{e11}), the iPALM \cite{pock2016inertial} type algorithm utilizes auxiliary function technology but cannot guarantee the monotonous decrease of the objective function, the BPL \cite{xu2017globally} type algorithm requires careful selection of extrapolation parameters and step size using the backtracking method, and both types of algorithms face challenges in requiring strict and complex restrictions on the step size and the selection range of extrapolation parameters, which will obviously lead to the decrease of the numerical performance of the algorithm. Compared with BPL\cite{xu2017globally}, iPALM\cite{pock2016inertial}, APGnc framework \cite{yao2016more,li2017convergence} has wider choice of extrapolation parameters and step size, but it also has some fixed defects: (\romannumeral1): APGnc framework only works for updating one variable block, but not multivariate optimization problems. (\romannumeral2): APGnc framework rarely has global convergence \cite{li2015accelerated,yao2016more,li2017convergence} due to its origin from APG \cite{beck2009fast}. (\romannumeral3): APGnc framework's extrapolated step is almost in a failure state. We can rigorously demonstrate mathematically that the extrapolation step of APGnc has a failure probability of 1 for some particular cases, more information can be found in the Appendix \ref{Appdix2}.

For a solution to problem (i) of APGnc, an intuitive approach we consider is Gauss-Seidel method \cite{bolte2014proximal} combined with it. For a solution to problem (ii) of APGnc, we show the following theorem in order to satisfy the requirement that the extrapolation sequence under problem (\ref{e11}) be meaningful: 
\begin{theorem}
Let: $ U(z)=prox_{\sigma_{j} F_{j}} (z-\sigma_{j} \nabla_{x_{j}} h(z)).$. With Eq. (\ref{e32}), when the $x_{(n)}^{0} \in S=\mathrm{dom} \ J$, then $x_{(n)}^{k} \in S~({\forall} k \in \mathbb{N})$
\label{algo_flaw}
\end{theorem}
\begin{proof}
It is proved by contradiction:  if ${\exists} j,k \in \mathbb{N}$, $x^{k+1}_{j} \notin S$, then at $x^{k+1}_{j}$, the function is undefined, i.e. $h(x^{k+1}_{j})= \infty$ or $F_{j}(x^{k+1}_{j})= \infty$, then from Assumption \ref{assump1}, ${\exists} a \in S$, s. t.: $U(a)<
\infty$, then $U(a)<U(x^{k+1}_{j})$, which contradicts that $x^{k+1}_{j}$ is the minimum point on Eq. (\ref{e32}). 
\end{proof}
For a solution to problem (iii) of APGnc, 
% one of the most crucial requirements for the algorithm framework is to be able to complete the first two steps of the Property \ref{based}. As a result, 
We need to design a completely different fundamental algorithm framework in order to complete the first two steps of Property \ref{based} as well as the requirements for moving on to the third step, while maintaining the APGnc framework's excellent specific qualities.

Based on our proposed solution, we propose an accelerated proximal alternating linear minimization algorithm framework with adaptive momentum (APBL$^+$) to solve problem (\ref{e11}). 

\begin{algorithm}[!htbp]
\caption{\small ABPL$^+$: Randomized/deterministic accelerated block proximal linear method with adaptive momentum}\label{ABPL$^+$}
  \begin{small}
    \SetAlgoLined
     \LinesNumbered
        % Input
        \KwIn{$\left\{{x^{1}_{i}} \right\}^{N}_{i=1}=\left\{{x^{0}_{i}} \right\}^{N}_{i=1} \in \mathrm{dom} \ J$,  $ k_{\max}=c, ~ t\in \left\{(1, \infty)\right\}_{i=1}^{N} , \beta_{\max} \in [0, 1),~ \beta_{1} \in \left\{[0, \beta_{\max}]\right\}_{i=1}^{N}$.}
        
        % Output
        \KwOut{$\left\{{x^{k+1}_{i}} \right\}^{N}_{i=1}$.}
        
        \For {$k = 1$ to $k_{\max}$}{
        $\left\{{y^{k}_{i}} \right\}^{N}_{i=1}$=$\left\{{x^{k}_{i}} \right\}^{N}_{i=1}$+$\beta_{k} (\left\{{x^{k}_{i}} \right\}^{N}_{i=1}-\left\{{x^{k-1}_{i}} \right\}^{N}_{i=1})$.\\
        $S=\left[ N \right], ~ T_{s}=\emptyset, ~ \left\{g_{i}\right\}_{i=1}^{N}=\left\{y^{k}_{i}\right\}_{i=1}^{N}$.
        
        \Repeat{$S=\emptyset$}{
	   Choose $j\in S$ deterministically or randomly.\\
         $S=S-\left\{j\right\},~\sigma_{j}^{k} \in (0, \frac{1}{L_{\nabla H_{x^{k}_{j}}}})$,\\ 
         $h(y_{j}^{k})=H(\left\{g_{i}\right\}_{i\in T_{s}\cup S, i<j}, y_{j}^{k}, \left\{g_{i}\right\}_{i\in T_{s}\cup S, i>j})$,\\
         $z^{k+1}_{j} \in prox_{\sigma_{j}^{k} F_{j}} (y^{k}_{j}-\sigma_{j}^{k} \nabla_{x_{j}} h(y_{j}^{k}))$.\\
         $g_{j}=z^{k+1}_{j},~T_{s}=T_{s}\cup \left\{j \right\}$.
        }

        $\left\{{x^{k+1}_{i}} \right\}^{N}_{i=1}=\left\{{z^{k+1}_{i}} \right\}^{N}_{i=1}$. 
        
        \eIf{\begin{equation}\label{if}J(\left\{{y^{k}_{i}} \right\}^{N}_{i=1}) \leq J(\left\{{x^{k}_{i}} \right\}^{N}_{i=1})\end{equation}} 
        {$\beta_{k+1}=\min(\beta_{max},t*\beta_{k})$.}{
        $\left\{{y^{k}_{i}} \right\}^{N}_{i=1}=\left\{{x^{k}_{i}} \right\}^{N}_{i=1}$, and get $\left\{z^{k+1}_{i}\right\}^{N}_{i=1}$ again through step 3 to 10,
        
        \eIf{\begin{equation}\label{ifo}J(\left\{{z^{k+1}_{i}} \right\}^{N}_{i=1}) > J(\left\{{x^{k+1}_{i}} \right\}^{N}_{i=1})\end{equation}}
        {
            $\beta_{k+1}=\min(\beta_{max}, t*\beta_{k})$.
        }
        {
            $\left\{{x^{k+1}_{i}} \right\}^{N}_{i=1}=\left\{{z^{k+1}_{i}} \right\}^{N}_{i=1}$, ~$\beta_{k+1}=\frac{\beta_{k}}{t}$.
        }
        
        }
        }
    \end{small}
\end{algorithm}

\begin{remark}
(\romannumeral1) When the $\beta_{k} \equiv 0$, the Algorithm \ref{ABPL$^+$} is ABPL$^+$ with no acceleration. Compared with PALM \cite{bolte2014proximal}, ABPL$^+$ without acceleration allows random shuffling.
% (\romannumeral2) $\beta_{k}$ can be individually customized for each variable block, as long as the range belongs to [0, \betamax], for example: $\beta_{k}=\left\{\beta_{k_{j}}\right\} \in \left\{[0,\beta_max)}$

(\romannumeral2) For the step 13, 17 and 19 of the ABPL$^+$, ABPL$^+$ is still successful if the extrapolation parameters do not follow an adaptive strategy, such as $t_{k+1}=\frac{1+\sqrt(1+4*t_{k}^2)}{2}, ~ \beta_{k+1}=\frac{t_{k+1}-1}{t_{k}}$. 
This extrapolated parameter update strategy comes from  \cite{beck2009fast}, and our algorithm is also valid for these strategies. 
% (\romannumeral3) If we use ABPL$^+$ to solve some sparse optimization with $\ell_0$-norm constrain, then we can delete the twelfth to fourteenth steps of the algorithm.

(\romannumeral3) When the extrapolated parameters are not adaptively updated, we refer to ABPL$^+$ as ABPL.
\end{remark}

Algorithm \ref{ABPL$^+$} clearly demonstrates that our algorithm can be updated in a random order. The step size $\sigma$ and the extrapolation parameter $\beta$ are independent of one another in the ABPL$^+$, thus there is no complicated or onerous choice limitation. 
% In the next section, we will rigorously prove that these advantages in ABPL$^+$ are effective, and can strictly guarantee the establishment of the Property \ref{based} without using auxiliary functions (that is, the objective function must be monotonically decreasing).

It is worth noting that Eq. (\ref{if}) should be preserved to improve the performance of the algorithm. There are several reasons, (\romannumeral1) first, when $\beta_{k} \equiv 0$, then $\left\{{y^{k}_{i}} \right\}^{N}_{i=1} \equiv \left\{{x^{k}_{i}} \right\}^{N}_{i=1}$, thus Eq. (\ref{if}) is always established. (\romannumeral2) Second, our algorithm can be proven to be globally convergent in Section \ref{Convergence}, so the sequence generated by our algorithm is a Cauchy sequence. (\romannumeral3) Third, although as shown in the Appendix \ref{Appdix2}, some specific problems will cause the measure of the point set where Eq. (\ref{if}) holds to be 0, but the zero measure set does not mean that it is an empty set, that is, Eq. (\ref{if}) is not an impossible event, just the probability is 0. For example, consider the Cantor set $\mathcal{C}$, its cardinality (Definition \ref{dcard}) $card(\mathcal{C})=card(\mathbb{R})$, or for example the set of rational numbers, their probability measure is both 0 (see \cite{ciarlet2013linear,halmos2013measure,stein2009real} for more information). 

\section{Convergence Analysis}
\label{Convergence}
In this section, we will describe the main steps to realize the Property \ref{based}. 
\subsection{Sufficient decrease of the objective function}
We will prove the Property \ref{based} (\romannumeral1). First of all, we will extend the number of variables to be updated by the ABPL$^+$ algorithm to any positive integer, and the update order can be random, as well as the objective function has the property of monotonically decreasing under Algorithm \ref{ABPL$^+$}.
\begin{theorem}
suppose that Assumption \ref{assump1} hold, let $x_{(N)}$ are sequences generated by Algorithm \ref{ABPL$^+$}. \\ 
(\romannumeral1):
$J(x^{k+1}_{(N)})$ is nonincreasing and in particular:
\begin{equation}
\begin{aligned}
    J(x^{k+1}_{(N)})\leq
    J(x^{k}_{(N)}) -\rho||x^{k+1}_{(N)}-y^{k}_{(N)}||, \\
\end{aligned}
\label{e33}
\end{equation}
where $\rho>0$, and $\rho$ is a positive constant.  \\
(\romannumeral2): we have: 
\begin{center}
$\lim_{k \to \infty} ||x^{k+1}_{(N)}-y^{k}_{(N)}||=0$.
\end{center}
\label{c1}
\end{theorem}
\begin{proof}
(\romannumeral1): Set $S=\left[ N \right]$, 
% from natural induction and Defintion \ref{dcard}, set $T_{s} \subseteq S$, when $card(T_{s})=1$, let $j \in T_{s}$, we want: 
% \begin{center}
%     $J(\left\{y_{i}^{k}\right\}_{i=1}^{j-1}, x_{j}^{k+1}, \left\{y_{i}^{k}\right\}_{i=j+1}^{N})\leq
%     J(y^{k}_{(N)}) -\rho||x_{j}^{k+1}-y_{j}^{k}||$,
% \end{center}
% where $\rho_{j}=\frac{1}{2\sigma_{j}}-L_{\nabla_{x_{j}} H}$,
$T_{s} \subseteq S,~card(T_{s})=1,~j \in T_{s}$, since Proposition \ref{p3} and Algorithm \ref{ABPL$^+$}, we have: 
\begin{align}
&& h(x_{j}^{k+1})&=H(\left\{y_{i}^{k}\right\}_{i=1}^{j-1}, x_{j}^{k+1}, \left\{y_{i}^{k}\right\}_{i=j+1}^{N}). \notag \\
&&h(x_{j}^{k+1})&\leq H(y^{k}_{(N)})+\langle \nabla_{x_{j}} H(y^{k}_{(N)}), x_{j}^{k+1}-y^{k}_{j}\rangle \notag\\ 
&&\ &+\frac{L_{\nabla_{x_{j}} H}}{2}||x^{k+1}_{j}-y^{k}_{j}||. 
\label{e331}
\end{align}
according to Eq. (\ref{e32}), we obtain: 
\begin{align}
 && F_{j}(y^{k}_{j}) &\geq \langle \nabla_{x_{j}} h(x_{j}^{k+1}), x_{j}^{k+1}-y^{k}_{j} \rangle+F_{j}(x^{k+1}_{j}) \notag \\
 &&\  &+\frac{\sigma_{1}}{2}||x_{1}-y^{k}_{1}||^{2}. 
 \label{e332}
\end{align}
then sum of Eq. (\ref{e331}) and Eq. (\ref{e332}), we have: 
\begin{align}
&& H(y^{k}_{(N)})+F_{j}(y_{j}^{k+1}) &\geq F_{j}(x_{j}^{k+1})+\rho_{j}||x^{k+1}_{j}-y^{k}_{j}||^{2} \notag \\
&& &+H(\left\{y_{i}^{k}\right\}_{i=1}^{j-1}, x_{j}^{k+1}, \left\{y_{i}^{k}\right\}_{i=j+1}^{N}). 
\end{align}
where $\rho_{j}=\frac{1}{2\sigma_{j}}-L_{\nabla_{x_{j}} H}$. Therefore, when $card(T_{s})=1$, Eq. (\ref{e33}) is obviously true. Assuming the Eq. (\ref{e33}) holds when $card(T_{s})=n$, i.e. 
\begin{align}
&& &H(\left\{x_{i}^{k+1}\right\}_{i=1}^{m-1}, y_{m}^{k}, \left\{x_{i}^{k+1}\right\}_{i=m+1}^{n+1}, \left\{y_{i}^{k}\right\}_{i=n+1}^{N}) \notag \\
&& &\leq H(y^{k+1}_{(N)})+\sum_{i\in T_{s}} F_{i}(y_{i}^{k+1})-\sum_{i\in T_{s}} F_{i}(x_{i}^{k+1}) \notag \\
&& &-\rho||\left\{x_{i}^{k+1}\right\}_{i=1, i\neq m}^{n+1}-\left\{x_{i}^{k+1}\right\}_{i=1, i\neq m}^{n+1}||. \label{e333}
\end{align}
$\rho=\min(\left\{\frac{1}{2\sigma_{i}}-L_{\nabla_{x_{i}} H}
\right\}_{i=1,i\neq m}^{N})$, we hope that is still established when $card(T_{s})=n+1$. Since Algorithm \ref{ABPL$^+$}, there are some inequalities also hold($m \neq i$): 
\begin{align}
&&h(y^{k}_{m})&= H(\left\{x_{i}^{k+1}\right\}_{i=1}^{m-1}, y_{m}^{k}, \left\{x_{i}^{k+1}\right\}_{i=m+1}^{n+1}, \left\{y_{i}^{k}\right\}_{i=n+1}^{N}).\notag \\ 
&& F_{m}(y^{k}_{m}) &\geq F_{m}(x^{k+1}_{m}) +\frac{\sigma_{m}}{2}||x_{m}-y^{k}_{m}||^{2} \notag \\ 
&&   &+\langle \nabla_{x_{m}} h(y^{k}_{m}), x_{m}^{k+1}-y^{k}_{m} \rangle. \label{e334}
\end{align}
from Proposition \ref{p3}, we infer: 
\begin{align}
&& H(x^{k+1}_{(n+1)}, \left\{y_{i}^{k}\right\}_{i=n+1}^{N}) &\leq h(y^{k}_{m})+\frac{L_{\nabla_{x_{j}} H}}{2}||x^{k+1}_{j}-y^{k}_{j}||  \notag \\
&& &+ \langle \nabla_{x_{j}} h(y^{k}_{j}), x^{k+1}_{j}-y^{k}_{j}\rangle .\label{e335}
\end{align}
thus, sum of Eq. (\ref{e333}), Eq. (\ref{e334}) and Eq. (\ref{e335}), natural induction teaches us that when $n+1=N$, that: 
\begin{align}
&& &J(\left\{{x^{k}_{i}} \right\}^{n}_{i=1}) \leq J(\left\{{y^{k}_{i}} \right\}^{n}_{i=1})-\rho||x^{k+1}_{(N)}-y^{k}_{(N)}||.
\label{c1end}
\end{align}
where $\rho=\min(\left\{\frac{1}{2\sigma_{i}}-L_{\nabla_{x_{i}} H}
\right\}_{i=1}^{N})$, if Eq. (\ref{if}) is true, then Eq. (\ref{e33}) is holds. 

If Eq. (\ref{ifo}) is true, from Eq. (\ref{c1end}) and $y^{k}_{(N)}=x^{k}_{(N)}$, we know that:
\begin{center}
    $J(z^{k+1}_{(N)}) \leq J(x^{k}_{(N)}) -\rho_{t}||z^{k+1}_{(N)}-x^{k}_{(N)}||$.
\end{center}
where $\rho_{t}=\min(\left\{\frac{1}{2\sigma_{i}}-L_{\nabla_{x_{i}} H}
\right\}_{i=1}^{N})$, thus we also have: 
\begin{center}
$J(x^{k+1}_{(N)}) < J(x^{k}_{(N)}) -\rho_{t}||z^{k+1}_{(N)}-y^{k}_{(N)}||$.
\end{center}
if $||z^{k+1}_{(N)}-y^{k}_{(N)}||\leq ||x^{k+1}_{(N)}-y^{k}_{(N)}||$, define the set $T_{k}=\left\{ J(x^{k+1}_{(N)})- J(z^{k+1}_{(N)}) \ | \ Eq.~(\ref{ifo}) \ is \ ture)\right\}$, from Proposition \ref{zorn}, \ref{natureorder}, there $\exists M \in T_{k}$, $ {\exists\mkern-10.5mu/}$ $T\in T_{k}$, s.t $M\neq T,~M \leq T $, we thus obtain:
\begin{align}
J(x^{k+1}_{(N)}) \leq J(x^{k}_{(N)}) -\rho_{t}||z^{k+1}_{(N)}-y^{k}_{(N)}||-|M|.
\label{c1ifo}
\end{align}
$||x^{k+1}_{(N)}-y^{k}_{(N)}||$ is bound since Assumption \ref{assump1}, thus define $M_{T}=\sup \left\{||x^{k+1}_{(N)}-y^{k}_{(N)}||\right\}$, we also have: 
\begin{center}
$J(x^{k+1}_{(N)}) \leq J(x^{k}_{(N)}) -\rho_||x^{k+1}_{(N)}-y^{k}_{(N)}||$.
\end{center}
% where  $\rho=\min \left\{\frac{|M|+\rho_{t}\inf \left\{||z^{k+1}_{(N)}-y^{k}_{(N)}||\right\}}{M_{T}},\rho_{t} \right\}>0$.\\
where  $\rho=\min \left\{\frac{|M|+\rho_{t}(||z^{k+1}_{(N)}-y^{k}_{(N)}||)}{M_{T}},\rho_{t} \right\}>0$.\\
(\romannumeral2): Since J is lower bounded, sum Eq. (\ref{e33}), i.e. 
\begin{center}
$\sum_{k=1}^{\infty} ||x^{k+1}_{(N)}-y^{k}_{(N)}||^{2}=J(x_{(N)}^{1})-\inf J<\infty$,
\end{center}
which imply $\lim_{k \to \infty} ||x^{k+1}_{(N)}-y^{k}_{(N)}||=0$. 
\end{proof}

\subsection{Subgradient lower bound for the iterates gap}
We will prove the Property \ref{based} (\romannumeral2) that the derivative set of J is the critical point set. We will make a more intuitive proof in the general proof framework of the PALM algorithm, and give some interesting properties of by-products. %is a novel scheme. 
% Compared with other APALM algorithms \cite{bolte2014proximal, pock2016inertial, xu2017globally, li2017convergence, wang2023generalized, le2020inertial, phan2023inertial, gao2020gauss, li2017convergence}, 
% which is more concise and intuitive. %is a novel scheme. 
\begin{lemma}
In Algorithm \ref{ABPL$^+$}, we have
\begin{equation}
\begin{aligned}
&& &\nabla_{x_{j}}h (x^{k+1}_{j}) -\nabla_{y_{j}}h (y^{k}_{j})+ \frac{1}{\sigma_{j}}  ||y^{k}_{j}-x^{k+1}_{j} || \notag \\ 
&& &\in \partial_{x_{j}}J(\left\{g_{i}\right\}_{i\in T_{s}\cup S, i<j}, y_{j}^{k}, \left\{g_{i}\right\}_{i\in T_{s}\cup S, i>j}).
\end{aligned}
\label{e34}
\end{equation}
\label{c2}
\end{lemma}
\begin{proof}
By Propositon \ref{p1} and Eq. (\ref{e32}), ${\forall} j \in n, ~ j \leq N$, it follows that: 
\begin{align}
&& 0 &\in  \nabla_{y_{j}}h(y^{k}_{j})-\frac{1}{\sigma_{j}} (y^{k}_{j}-x^{k+1}_{j})  \notag \\
&&\ & +\partial_{x_{j}} (\sum_{i \in T_{s}-\left\{j\right\}} F_{i}(x^{k+1}_{i})+\sum_{i \in S} F_{i}(y^{k}_{i})+F_{j}(x^{k+1}_{j})), \notag
\end{align}
which imply: \\
\begin{align}
&& &\nabla_{x_{j}}h (x^{k+1}_{j}) -\nabla_{y_{j}}h (y^{k}_{j})+ \frac{1}{\sigma_{j}}  ||y^{k}_{j}-x^{k+1}_{j} || \notag \\ 
&&\ &\in \nabla_{x_{j}}h (x^{k+1}_{j})+\partial_{x_{j}} (\sum_{i \in T_{s}} F_{i}(x^{k+1}_{i})+\sum_{i \in S} F_{i}(y^{k}_{i})) \notag\\
&&\ &=  \partial_{x_{j}}J(\left\{g_{i}\right\}_{i\in T_{s}\cup S, i<j}, y_{j}^{k}, \left\{g_{i}\right\}_{i\in T_{s}\cup S, i>j}). \notag
\end{align}
\\
which means the Lemma \ref{c2} is true for ${\forall}j \in \mathbb{N}$, i.e Eq. (\ref{e34}) . 
\end{proof}
\begin{theorem}
If we define
\begin{center}
$p_{x_{j}^{k+1}}=\nabla_{x_{j}}h (x^{k+1}_{j})-\nabla_{y_{j}}h (y^{k}_{j})+\frac{1}{\sigma_{j}
}||y^{k}_{j}-x^{k+1}_{j}||$.
\end{center}
Then we have:
\begin{equation}
\begin{aligned}
       ||\left\{p_{x^{k+1}_{i}} \right\}^{N}_{i=1}|| \leq \rho_{b}||(\left\{{x^{k+1}_{i}}-{y^{k}_{i}}\right\}^{N}_{i=1}||,
\end{aligned}
\label{e35}
\end{equation}    
where $\rho_{b}=max(\left\{\frac{1}{2\sigma_{i}}+L_{\nabla_{x_{i}} H}\right\}_{i=1}^{N})$.
\label{c3}
\end{theorem}

\begin{proof}
Set $S=\left[ N \right]$, from Defintion \ref{d3} and Assumption \ref{assump1}, when $card(T_{s})=1$, ${\forall}j \in \left[ N \right]$, we have: \\
\begin{align}
&& ||p_{x_{j}^{k+1}}|| &\leq \frac{L_{\nabla_{x_{j}} H}}{2}||x^{k+1}_{j}-y^{k}_{j}||+\frac{1}{2\sigma_{j}} ||x^{k+1}_{j}-y^{k}_{j}|| \notag \\
&& &=(\frac{L_{\nabla_{x_{j}} H}}{2}+\frac{1}{2\sigma_{j}}) ||x^{k+1}_{j}-y^{k}_{j}||.
\label{c31}
\end{align}

Next, assuming the original formula holds when $card(T_{s})=n$, i.e. 
\begin{align}
||\left\{p_{x^{k+1}_{i}} \right\}_{i\in T_{s}}|| \leq \rho_{b} ||(x^{k+1}_{i\in T_{s}}-y^{k}_{i\in T_{s}})|| .
\label{c32}
\end{align}
when $card(T_{s})=n+1$, let $j \in S-T_{s}$, sum of Eq. (\ref{c31}) and Eq. (\ref{c32}), since natural induction teaches us that when $n+1=N$, Eq. (\ref{e35}) is ture. 
\end{proof}
In metric space, some intriguing characteristics that are not reliant on the particular form of the algorithm can be given. 
% For example,we can simply and intuitively prove that the aggregation point sets of all sequences generated by the PALM and APALM algorithms are compact sets, i.e.: 
\begin{theorem}
For any algorithm based on a finite-dimensional metric space, let A be a set of all sequences of the algorithm, if A be a bounded set, the derivative set A' of A must be a compact set. \label{c4}
\end{theorem}
\begin{proof}
From Definition \ref{dopen}, $A' \subseteq \bar{A}$. Under Proposition \ref{dbouned}, $A'$ is a bounded set, then since Proposition \ref{dtop} and \ref{dcompact}, the Theorem is established. 
\end{proof}
% \myspace{0}It can be seen that as long as Assumption \ref{assump1} is satisfied, no matter what the specific metric space form of the algorithm is, the derivative set of the sequence generated by the algorithm is a compact set, thus: 
\begin{corollary}
Define A as the derivative set of the set composed of all sequences generated by Algorithm \ref{ABPL$^+$}, then A is a compact set. 
\end{corollary}
\begin{proof}
It is obvious from Theorem \ref{c4}. 
\end{proof}
Let $z^{k}=\left\{x^{k}_{i} \right\}^{N}_{i=1}, ~ w^{k}=\left\{y^{k}_{i} \right\}^{N}_{i=1}$, from Definition \ref{d1}, define $z'$ is the derivative set of $\left\{z\right\}$. 
% So far we have obtained the iterative value distance inequality and function value descent inequality. 
Next, we will use the Lemma \ref{tf4} to prove: 

\begin{theorem}
Let $z^{k}$ be a sequence generated by Algorithm \ref{ABPL$^+$}. Then J is a constant on z' and $z' \subseteq crit J$.
\label{c5}
\end{theorem}
\begin{proof}
According to Assumption \ref{assump1}, ${\forall}\overline{z} \in z'$, there exists a subsequence $z^{k_{j}}$  such that: $\lim_{j \to \infty} z^{k_{j}}= \overline{z}$, since Lemma \ref{tf4}, $\overline{z}$ of arbitrariness and Theorem \ref{c1}, we infer: 
  \begin{center}
   $\lim_{j \to \infty} J(z^{k_{j}})=J( \overline{z})=J^{*}$.
    \end{center}
which means J is a constant on z'. Since Theorem \ref{c1} and Theorem \ref{c3}, we have: 
  \begin{center}
   $\lim_{k \to \infty} ||\left\{p_{x^{k+1}_{i}} \right\}^{N}_{i=1}|| \leq \lim_{k \to \infty} \rho_{b}||(\left\{{x^{k+1}_{i}}-{y^{k}_{i}}\right\}^{N}_{i=1}||, $\\
   $\lim_{k \to \infty}||\left\{p_{x^{k+1}_{i}} \right\}^{N}_{i=1}|| = 0$.
    \end{center}
i.e. $z' \subseteq crit J$. 
\end{proof}
As a result, rather than being fundamental to the PALM algorithm framework, these inherent properties associated to the derived set are inherent properties in the finite dimensional metric spaces.

\subsection{Global convergence under KŁ Property}
We will prove Property \ref{based} (\romannumeral3):  the sequence generated by Algorithm \ref{ABPL$^+$} has global convergence (see Theorem \ref{glo}). 
Additionally, we also give the convergence rate of Algorithm \ref{ABPL$^+$} (see Theorem \ref{rate}). We point out that because our algorithm allow for a randomized strategy and the presence of extrapolated sequences in Property \ref{based} (\romannumeral1) for our algorithm, a straight application of the methodology cited in \cite{bolte2014proximal} to our suggested algorithms is not feasible. In the subsequent theorem, we adapt the proof strategy of to make it compatible with the algorithm we provide. 
\begin{theorem}
The\ sequence\ $\left\{z^{k}\right\}$ generated by Algorithm \ref{ABPL$^+$} is converge when $\beta \in \left\{[0, \beta_{\max}]\right\}_{i=1}^{N}$, i.e: $\lim_{k\rightarrow \infty} ||z^{k+p}-z^{k}||=0(\forall p\in \mathbb{N})$
\label{glo}
\end{theorem}
\begin{proof}
According to Assumption \ref{assump1}, $\left\{z^{k}\right\}$ is bounded and complete. From Definition \ref{d8}, ${\forall} \eta >0$, there exists a positive integer $k_{0}$ such that: $J( \overline{z})<J(z^{k_{0}})<J( \overline{z})+\eta$. Since Definition \ref{d8} and Theorem \ref{c3}, there exists a concave function $\phi$ so that: $ \phi^{'}(J(z^{k})-J( \overline{z})) dist(0, \partial J(z^{k})) \geq 1$, thus we infer: 
\begin{align}
&& dist(0, \partial J(z^{k}) ) &\leq ||\left\{p_{x^{k+1}_{i}} \right\}^{N}_{i=1}||=\rho_{b}||z^{k}-w^{k-1}||. \notag
% && &\leq \rho_{b}|(\left\{{x^{k+1}_{i}}-{y^{k}_{i}}\right\}^{N}_{i=1}|| \notag \\
% && &=\rho_{b}||z^{k}-w^{k-1}||. \notag
\end{align}
according to the $ \phi^{'}(J(z^{k})-J( \overline{z})) dist(0, \partial J(z^{k})) \geq 1$, which imply: 
  \begin{center}
  $ \phi^{'}(J(z^{k})-J( \overline{z})) \geq \frac{1}{dist(0, \partial J(z^{k}))}$
  $\geq \frac{1}{\rho_{b}||z^{k}-w^{k-1}||}$.
    \end{center}
let $G(k)=J(z^{k})-J( \overline{z})$, from definition of concave function and Theorem \ref{c1}, we have: 
  \begin{align}
&& \phi(G(k))-\phi(G(k+1)) &\geq \phi^{'}(G(k))(G(k)-G(k+1))  \notag \\
&& &\geq \frac{\rho||z^{k+1}-w^{k}||^{2}}{\rho_{b}||z^{k}-w^{k-1}||}. \notag 
    \end{align}
define C=$\frac{\rho}{\rho_{b}}$, C is a constant, we infer: 
  \begin{center}
    $||z^{k+1}-w^{k}||^{2} \leq C (\phi(G(k))-\phi(G(k+1))) ||z^{k}-w^{k-1}||$.
    \end{center}
Using the fact that $2ab\leq a^{2}+b^{2}$: 
  \begin{center}
    $2 ||z^{k+1}-w^{k}|| \leq C (\phi(G(k))-\phi(G(k+1))) +||z^{k}-w^{k-1}||$.
    \end{center}
sum both sides: 
\begin{align}
&&  2 \sum_{k=l+1}^{K} ||z^{k+1}-w^{k}|| &\leq \sum_{k=l+1}^{K} ||z^{k}-w^{k-1}|| \notag \\
&& &+C(\phi(G(l+1))-\phi(G(K+1))) \notag 
% && &=C (\phi(G(l+1))-\phi(G(K+1))) \notag \\
% && &+||z^{l+1}-w^{l}||+\sum_{k=l+1}^{K}||z^{k+1}-w^{k}||. \notag 
\end{align}
thus from Lemma \ref{tf4} and Assumption \ref{assump1}, we get that: 
\begin{align}
&& \lim_{K \to \infty} \sum_{k=l+1}^{K} ||z^{k+1}-w^{k}|| &\leq ||z^{l+1}-w^{l}|+ C\phi(G(l+1))  \notag \\
&& &-\lim_{K \to \infty} C \phi(G(K+1)))
% && &= C (\phi(G(l+1))+||z^{l+1}-w^{l}||   \notag \\
% && &-\phi(\lim_{k \to \infty} G(K+1)))  \notag \\
 \label{c7}
\end{align}
from Assumption \ref{assump1}, no matter $w^{k} =z^{k}+\beta_{k} (z^{k}-z^{k-1})$ or $w^{k} =z^{k}$, we always have: 
  \begin{center}
    $ \ ||z^{k+1}-z^{k}|| -\beta_{max} ||z^{k}-z^{k-1}||  \
  \leq||z^{k+1}-w^{k}|| $.
    \end{center}
let $s_{k}=||z^{k+1}-z^{k}||$, from Eq. (\ref{c7}), we know that $\sum_{k=K}^{\infty}( s_{k+1} -\beta_{max} s_{k})<\infty$, thus we infer: 
 \begin{align}
&&\sum_{k=K+1}^{\infty} (1-\beta_{max})s_{k}-\beta_{max} s_{K} &=\sum_{k=K}^{\infty}( s_{k+1} -\beta_{max} s_{k}) \notag 
% && &<\infty. \notag
\end{align}
from  $(1-\beta_{max})s_{k}>0$ and $ (1-\beta_{max})$ is a constant, thus: 
    \begin{align}
    &&\lim_{K\rightarrow \infty} ||z^{K+p}-z^{K}|| \leq \lim_{K\rightarrow \infty}\sum_{k=K+1}^{\infty}s_{k}=0.
    \end{align}
which means the Theorem \ref{glo} is true. 
\end{proof}
Under KŁ inequality, we can obtain convergence rate results as following: 
\begin{theorem} 
(Convergence rate): set Assumption \ref{assump1} is true, let $x^{k}_{(N)}$ be a sequence generated by Algorithm \ref{ABPL$^+$}, the desingularizing function has the form of $\phi(t) = \frac{\theta}{C}t^{\theta}$, with $\theta\in (0, 1], ~ c > 0$. Let $J^{*}=J(e)(e \in z'), ~ r^{k} = J(x^{k}_{(N)})-J^{*}$. The following assertions hold: \\
(\romannumeral1): If $\theta =1$, the Algorithm \ref{ABPL$^+$} terminates in finite steps. \\
(\romannumeral2): If $\theta \in [\frac{1}{2}, 1)$, then there exist a integer $k_{2}$ such that
\begin{center}
$r^{k}\leq (\frac{d_{1}C^{2}}{1+d_{1}C^{2}}),{\forall k_{2} \geq k}$.
\end{center}
(\romannumeral3): If $\theta \in (0, \frac{1}{2})$, then there exist a integer $k_{3}$  such that
\begin{center}
$r^{k}\leq [\frac{C}{(k-k_{3})d_{2}(1-2\theta)}]^{\frac{1}{1-2\theta}},{\forall k_{3} \geq k}$.
\end{center}
where 
\begin{center}
$d_{1}=(\frac{\rho_{b}}{\rho})$, $d_{2}=\min\left\{\frac{1}{2d_{1}C}, \frac{C}{1-2\theta}(2^{\frac{2\theta-1}{2\theta-2}})r_{0}^{2\theta-1}\right\}$.
\end{center}
\label{rate}
\end{theorem}
\begin{proof}
Checking the assumptions of Theorem 2 in reference \cite{li2017convergence}, we observe that all assumptions required in Algorithm \ref{ABPL$^+$} are clearly satisfied, so Theorem \ref{rate} holds. 
\end{proof}

\section{Numerical experiments} 
\label{section: B}
In this section, we will introduce some application. All algorithms run on this configuration: 
12th Gen Intel(R) Core(TM) i7-12700   2. 10 GHz, RAM	32. 0 GB (31. 8 GB available);64-bit operating system;realized on the configuration of Matlab 2021b. The fundamental tensor computation was based on Tensor Toolbox 3.5 \cite{bader2006algorithm}. The code is available at \url{https://github.com/Weifeng-Yang/ABPL}.

\subsection{Proposed Algorithms and Baseline Algorithms}
\label{subsection: asup}
Based on our proposed ABPL$^+$ (See Algorithm \ref{ABPL$^+$}),  we propose four transformation forms: 
%\begin{itemize}
%    \item [1)] ABPL-cyclic:  It is an accelerated block proximal algorithm with non-adaptive mometum and fixed cyclic order.    
%    \item [2)] ABPL-random:  It is an accelerated block proximal algorithm with non-adaptive mometum and random shuffling. 
%    \item [3)] ABPL$^+$-cyclic:  It is an accelerated block proximal algorithm with adaptive mometum and fixed cyclic order. 
%    \item [4)] ABPL$^+$-random:  It is an accelerated block proximal algorithm with adaptive mometum and random shuffling.  
%\end{itemize}

ABPL-cyclic:  It is an accelerated block proximal algorithm with non-adaptive momentum and fixed cyclic order;

ABPL-random:  It is an accelerated block proximal algorithm with non-adaptive momentum and random shuffling;

ABPL$^+$-cyclic:  It is an accelerated block proximal algorithm with adaptive momentum and fixed cyclic order;

ABPL$^+$-random:  It is an accelerated block proximal algorithm with adaptive momentum and random shuffling.  

We compare these algorithms with state-of-the-art algorithms for solving two non-convex and non-smooth problems. 
\begin{itemize}
    \item [1)] PALM  \cite{bolte2014proximal}:  Proximal Alternating Linearized Minimization (PALM) is a classical first-order method for nonconvex and nonsmooth problems. 
    
    \item [2)] APGnc$^+$  \cite{li2017convergence}:  Accelerated proximal gradient method for nonconvex programming(APGnc) with adaptive momentum. Through the restart step, the algorithm makes it unnecessary to strictly limit the extrapolation parameter $\beta$ and step size. 
    
    \item [3)] iPALM  \cite{pock2016inertial}:  The inertial Proximal Alternating Linearized Minimization (iPALM) method is an accelerated version of PALM. By severely limiting extrapolated parameters($\alpha$ and $\beta$) as well as the choice of step size, iPALM can satisfy the Property \ref{based} but is not a monotonically decreasing approach. 
    
    \item [4)] BPL  \cite{xu2017globally}:  Randomized/deterministic block prox-linear (BPL) method is a BCD-based first-order method for nonconvex and nonsmooth problems. By strictly restricting the selection of extrapolation parameters ($\beta$) and step size, and using the backtracking method to backtrack the extrapolation parameter values that meet the requirements, BPL can satisfy Property \ref{based} and make the objective function monotonically decreasing. 

    \item [5)] IBPG  \cite{le2020inertial}:  Inertial block proximal gradient (IBPG) method is a first-order accelerated algorithm. By adding a conditional such as strong convex, IBPG allowing repeating the update, By severely limiting extrapolated parameters($\alpha$ and $\gamma$) as well as the choice of step size, IBPG can satisfy the Property \ref{based} but is not a monotonically decreasing approach. 
\end{itemize}
we give the parameter update strategy of the extrapolated parameters of the algorithms in Table \ref{betaupdate} (applicable to the following two examples). 
\begin{table*}
\renewcommand\arraystretch{1.4}
\centering
\caption{Update strategy of the extrapolated parameters $\beta_{k}$ and the update condition of APGnc$^+$, iPALM, APBL, ABPL$^+$}\label{betaupdate}
\scalebox{1}{
\begin{tabular*}{\linewidth}{l|l|l} %  % 
\toprule     
Algorithm & Update strategy of extrapolated parameters  & Remark \\  %换行
\midrule   

iPALM \cite{pock2016inertial} & \thead[l]{$\beta_{t}=\frac{(i-1)}{(i+1)}, ~ \alpha_{t}=\frac{(i-1)}{2*(i+1)}$} & It is an accelerated version of PALM \cite{bolte2014proximal}. \\ \hline

% APGnc$^+$ \cite{li2017convergence} & \thead[l]{$\beta_{k+1}=\left\{
%     \begin{aligned}
%     && & \min(\beta_{max},t\beta_{k}),~if~ F(x_k)\leq F(v_k)\\
%     && & \frac{\beta_{k}}{t}, ~else \\
%     \end{aligned}
%     \right   . $} & $\beta_{max}<1,~t>1$;.It is an adaptive version of APGnc \cite{yao2016more} \\ \hline
    
ABPL-cyclic & \thead[l]{$t_{k+1}=\frac{1+\sqrt(1+4*t_{k}^{2})}{2}, ~ \beta_{k+1}=\frac{t_{k}-1}{t_{k+1}}$}  & $t_{1}=1$; This strategy comes from  \cite{beck2009fast} on solving convex.\\  \hline    

ABPL-random & \thead[l]{$t_{k+1}=\frac{1+\sqrt(1+4*t_{k}^{2})}{2}, ~ \beta_{k+1}=\frac{t_{k}-1}{t_{k+1}}$}  & $t_{1}=1$; This strategy comes from  \cite{beck2009fast} on solving convex.\\  \hline

ABPL$^+$-cyclic & \thead[l]{$\beta_{k+1}=\left\{
    \begin{aligned}
    && & min(\beta_{max}, t\beta_{k}), ~if~Eq.~\ref{if} \ or \ Eq.~\ref{ifo} \ is \ true\\
    && & \frac{\beta_{k}}{t}, ~else \\
    \end{aligned}
    \right. $}& $\beta_{max}<1, ~t>1$; Updating block-variables with fixed cyclic order. \\  \hline
    
ABPL$^+$-random & \thead[l]{$\beta_{k+1}=\left\{
    \begin{aligned}
    && & min(\beta_{max}, t\beta_{k}), ~if~Eq. ~\ref{if} \ or \ Eq.~\ref{ifo} \ is \ true\\
    && & \frac{\beta_{k}}{t}, ~else \\
    \end{aligned}
    \right. $}& $\beta_{max}<1, ~t>1$; Updating block-variables with random shuffing. \\  
\bottomrule      
\end{tabular*}
}
\end{table*}

For BPL  \cite{xu2017globally}, IBPG  \cite{le2020inertial} and PALM  \cite{bolte2014proximal} have been proven to analyze the convergence of multivariate optimization problems, so we will not go into details. For APGnc$^+$ \cite{li2017convergence}, in the proof section \ref{section: A} of convergence of ABPL$^+$, it can be easily seen that the APGnc of multi block is only a degenerated form of ABPL$^+$, and for iPALM  \cite{pock2016inertial}, we only assume that its convergence to multi block variables still holds.

\subsection{Multiple sparse non-negative matrix factorization (msNMF) with $\ell_0$-norm constrain }
\label{subsection: NMF}
The multilayer NMF model \cite{guo2019sparse} extends the typical NMF explicitly to multiple layers, it is a more complex non-negative matrix factorization problem. 

To illustrate the effectiveness of our algorithm, we will apply ABPL$^+$ to the multiple non-negative matrix factorization with $\ell_0$-norm constrain, which can be presented: 
\begin{equation}
\begin{aligned}
&&  &\min \limits_{\left\{X_{i}\right\}_{i=1}^{n}} \frac{1}{2}||A-\prod_i^N X_{i}||_{F}^{2},\\
&& s. t \:  X_{i} &\geq 0,~||X_{i}||_{0} \leq s_{i} \ ({\forall} i \in \left[ N \right], ~s_{i} \in \mathbb{N}). 
\end{aligned}
\label{e41}
\end{equation}
when $N=2$, Eq. (\ref{e41}) degenerates into a sparse non-negative matrix factorization problem with $\ell_0$-norm constraint ($\ell_0$-msNMF).

\subsubsection{Solving $\ell_0$-msNMF using ABPL$^+$}
% At first, we analyze the relevant properties of the msNMF problem. 
If we write Eq. (\ref{e41}) in the form of Eq. (\ref{e11}), then we have:
% Comparing with Eq. (\ref{e11}), then:  
\begin{center}
$H(\left\{{X_{i}} \right\}^{N}_{i=1})= \frac{1}{2}||A-\prod_i^N X_{i}||_{F}^{2}$, $F_{i}(X_{i})=\delta_{i}(X_{i})$,
\end{center}
where
\begin{equation}
\delta_{i}(X_{i})=\left\{
\begin{aligned}
0&, \quad X_{i}\geq 0,~||X_{i}||_{0} \leq s_{i}, \\
\infty&,\quad else.
\end{aligned}
\right. 
\label{delta}
\end{equation}
% For the partial derivative of the function $H$: 
% \begin{small}
% \begin{align}
% && &\nabla_{X_{i}} H(\left\{{X_{j}} \right\}^{N}_{j=1})=(\prod_{j=1}^{i-1} X_{j})^{T}(\prod_{j=1}^N X_{j}-A)(\prod_{j=i+1}^N X_{j})^{T}. \label{enabla1}
% \end{align}    
% \end{small}
Now we will find the local lipschitz constant for each variable: 
\begin{lemma}
% Eq. (\ref{enabla1})  means 
$\nabla_{X_{i}} H(\left\{{X_{i}} \right\}^{N}_{j=1})$ are lipschitz gradient continuous with: 
\begin{small}
\begin{align}
&&& L(X_{i})=||((\prod_{j=1}^{i-1} X_{j})^{T}  \prod_{j=1}^{i-1} X_{j})||_{F} \times ||(\prod_{j=i+1}^N X_{j}  (\prod_{j=i+1}^N X_{j})^{T})||_{F}. \notag
\end{align}
\end{small}
\label{test1}
\end{lemma}
\begin{proof}
Because 
\begin{small}
\begin{align}
&& &\nabla_{X_{i}} H(\prod_{j=1}^{i-1}X_{j}, B_{i}, \prod_{j=i+1}^N X_{j})-\nabla_{X_{i}} H(\prod_{j=1}^{i-1}X_{j}, C_{i}, \prod_{j=i+1}^N X_{j})  \notag \\
&& &=((\prod_{j=1}^{i-1} X_{j})^{T}\prod_{j=1}^{i-1} X_{j})(B_{i}-C_{i})((\prod_{j=i+1}^N X_{j})(\prod_{j=i+1}^N X_{j})^{T}).\notag
\end{align}
\end{small}
and the Frobenius norm is a self-consistent matrix norm, which conforms to the norm definition of a linear operator \cite{ciarlet2013linear,stein2009real}, thus we have:
\begin{small}
\begin{align}
&& &||\nabla_{X_{i}} H(\prod_{j=1}^{i-1}X_{j}, B_{i}, \prod_{j=i+1}^N X_{j})-\nabla_{X_{i}} H(\prod_{j=1}^{i-1}X_{j}, C_{i}, \prod_{j=i+1}^N X_{j})||_{F}  \notag \\
&& &\leq ||(\prod_{j=1}^{i-1} X_{j})^{T}\prod_{j=1}^{i-1} X_{j}||_{F} 
\times ||\prod_{j=i+1}^N X_{j}(\prod_{j=i+1}^N X_{j})^{T})||_{F}\notag \\
&& &\times ||(B_{i}-C_{i})||_{F}. \notag
\end{align}
\end{small}
from Definition \ref{dlipchitz}, the Lemma \ref{test1} is true.
% the original proposition is true. 
\end{proof}
\myspace{0}Therefore, it is easy to verify that msNMF with $\ell_0$-norm constrain satisfies Assumption \ref{assump1}, and the proximal operator corresponding to $\ell_0$-msNMF is: 
\begin{align} 
 && X_{i}^{k+1} &\in prox_{\frac{1}{c^{k}_{i}} R_{i}(X^{k}_{i})}(U^{k}_{i}) \notag \\
 && &\in \operatorname*{\arg\min} \limits_{A} \left\{\frac{c_{i}^{k}}{2}||A-U^{k}_{i}||_{F}^{2}+F_{i}(X^{k}_{i}) \right\} \notag \\  
&& &\in  \operatorname*{\arg\min} \limits_{A}
 \left\{\frac{c_{i}^{k}}{2}||A-U^{k}_{i}||_{F}^{2}: A\geq 0, ~ ||A||_{0}\leq s_{i}\right\}, \notag  
\end{align}
where: 
\begin{center}
$U^{k}_{i}=X^{k}_{i}-\frac{1}{c^{k}_{i}} \nabla_{X_{i}} H(\left\{{X^{k+1}_{j}} \right\}^{i-1}_{j=1}, X^{k}_{i}, \left\{{X^{k}_{j}} \right\}^{N}_{j=i+1})$,\\$~c_{i}^{k}=\gamma_{i}*L(X_{i}^{k}), ~\gamma_{i}>1$. 
\end{center}

\subsubsection{Numerical tests}
In this experiment, we test the algorithms on the the datasets of SuiteSparse\footnote{\url{https://sparse.tamu.edu/}} \cite{10.1145/2049662.2049663}: dataset12mfeatfactors10NN ($\mathbb{R}^{2000\times2000}$) and lpship12l ($\mathbb{R}^{1153\times5533}$). 
% which’s dimension are $2000\times2000$ and $1153\times5533$. 
For each dataset,  the dimensions of the decomposition can be listed as following: 
\myspace{0}Dataset12mfeatfactors10NN:
$size(\left\{{X_{i}} \right\}^{3}_{i=1})=[2000\times 500, ~500\times 1300, ~1300\times 2000]$.
\myspace{0}Lpship12l: $size(\left\{{X_{i}} \right\}^{4}_{i=1})=[1151\times 1200, ~ 1200\times 500, ~ 500\times 3000, ~ 3000\times 5533]$.  

In our algorithms, set each initial hyperparameter according to the initial conditions of the algorithm as: $t=1.1, ~ \beta_{1}=0.6, ~ \beta_{max}=0.9999, ~ \gamma_{i}=1. 5$.
We run the algorithms for a specific sparsity setting: the number of non-zero elements in each matrix cannot exceed $30\%$ of the total number of elements, The initial point is set to a uniformly distributed random sparse positive definite matrix. 

In order to show the effectiveness of the proposed algorithms, we will evaluate them from two aspects:
(1) we run all algorithms ten times with the same hyperparameters, each time using the different random initialization point, and let all algorithms run for the same length of time. For dataset12mfeatfactors10NN, the running time was set by 400 seconds, for lpship12l, the running time was set by 500 seconds. We plotted the curves of the average objective function values of dataset12mfeatfactors10NN and lpship12l as a function of time under different algorithms in Figure \ref{figmatrix}(A) and Figure \ref{figmatrix}(B) respectively. 
(2) For the second aspect, define $RelErr^{k}=\frac{||A-\prod_{i=1}^{N}X_{i}^{k}||_{F}}{||A||_{F}}$, using the following stopping criteria: 
\begin{align}
\Delta_{RelErr^{k}} =|RelErr^{k+1}-RelErr^{k}|<\epsilon.
\label{stop}
\end{align}

For dataset12mfeatfactors10NN, the maximum running time was set by 800 seconds. For lpship12l, the maximum running time was set by 1000 seconds. $\epsilon$ is uniformly set to 1e-4. 
We run all algorithms ten times with the same hyperparameters and use the different random initialization point each time. The test results of these algorithms are presented in Table \ref{mattime}.
% Table \ref{mattime} shows the average time and the values of various indicators when the running result of the algorithms reach the termination condition. 
% \begin{figure*}
% \centering
% \subfloat[Average objective function value on dataset12mfeatfactors10NN]{
% \begin{minipage}{0.5\linewidth}
%     \includegraphics[width=9cm, height=6cm]{image/Figure1.1.pdf}
% \label{figmfea}
% \end{minipage}
% }
% \subfloat[Average objective function value on lpship2l]{
% \begin{minipage}{0.5\linewidth}
%     \includegraphics[width=9cm, height=6cm]{image/Figure1.2.pdf}
% \label{figlp}
% \end{minipage}
% }
% \caption{Convergence performance of different algorithms on the dataset12mfeatfactors10NN and lpship12l datasets.The figure shows the average objective function value.A partial enlarged view of several curves with the fastest descending speed and the best effect is also be given,and the coordinate unit values and curve mark symbols of the partially enlarged figure are the same as those of the original figure}
% \end{figure*}

\begin{figure}[hbpt]
    \centering \includegraphics[width=1\linewidth]{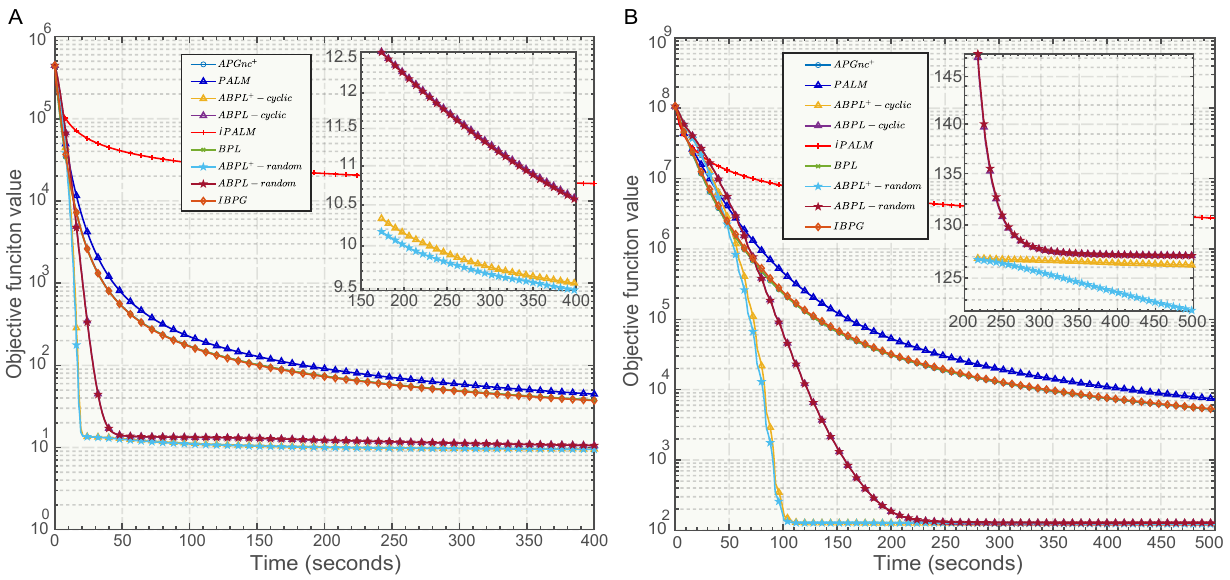}
    \caption{
    Comparison of the average convergence speed of the objective function of different algorithms on (A) the dataset12mfeatfactors10NN and (B) the lpship12l dataset. 
    The Convergence curves shows the average objective function value. A partial enlarged view of several curves with the fastest descending speed and the best effect is also be given, and the coordinate unit values and curve mark symbols of the partially enlarged figure are the same as those of the original figure. 
    }\label{figmatrix}
\end{figure}

\begin{table*}
\renewcommand\arraystretch{1.4}
\centering
\caption{Comparison of algorithms applied on the dataset12mfeatfactors10NN and lpship12l: the results are the average of ten runs and the bold indicates the best numerical performance under the set termination conditions.}
\scalebox{1}{
%\begin{tabular*}{\linewidth}{p{2cm}|p{2cm}|p{1.3cm}|p{1.5cm}|p{4.2cm}|p{1.4cm}|p{2.8cm}} %  % 
\begin{tabular*}{\linewidth}{p{2cm}|p{1.9cm}|p{1.7cm}|p{1.45cm}|p{4.2cm}|p{1.3cm}|p{2.6cm}} 
\toprule     
Data & Algorithm  &Objective function value & Time (seconds) &Final extrapolated parameters  & RelErr & $\Delta_{RelErr}$ (See Eq. \ref{stop}) \\
\midrule   
\multirow{9}{*}{dataset12mfea} & \multirow{1}{*}{PALM}  & 22.61 & 800 & $\sim$ & 3.276 & $8.22*10^{-4}$  \\ \cline{2-7}

\multirow{9}{*}{tfactors10NN \cite{10.1145/2049662.2049663}} & \multirow{1}{*}{APGnc$^+$} & 22.60 & 800 & $1.01*10^{-48}$ & 1.665 & $8.27*10^{-4}$ \\  \cline{2-7}

\multirow{9}{*}{} & iPALM &  12496 
& 800 & $\beta$=0.99, $\alpha$=0.49 &  918.184&  0.2862\\ \cline{2-7} 

\multirow{9}{*}{} & BPL  & 22.73 & 800 & [0.10, 0.10, 0.10, 0.10]   & 1.671 & $8.39*10^{-4}$ \\ \cline{2-7}

\multirow{9}{*}{} & IBPG  & 22.4415 & 800 & \thead[l]{$\gamma_{t}=[0.10, ~ 0.10, ~ 0.10]$ \\  $\alpha_{t}=[0.067, ~ 0.067, ~ 0.067]$}   & 1.650   & 0.0017 \\ \cline{2-7}

\multirow{9}{*}{} & \multirow{1}{*}{ABPL-cyclic}  & 9.05 & 800 & 0.99 & 0.777 & $1.95*10^{-4}$  \\ \cline{2-7}

\multirow{9}{*}{} & \multirow{1}{*}{ABPL-random}  & 9.04 & 800 & 0.99 & 0.776 & $1.92*10^{-4}$ \\ \cline{2-7}

\multirow{9}{*}{} & \multirow{1}{*}{ABPL$^+$-cyclic} &  9.52 & $\bm{353.21}$ & 0.99 & 0.702 & $9.75*10^{-5}$  \\ \cline{2-7}

\multirow{9}{*}{} & \multirow{1}{*}{ABPL$^+$-random} & 9.49 & $\bm{325.19}$ & 0.99  & 0.697 & $9.54*10^{-5}$  \\ \hline 

\multirow{9}{*}{lpship12l} & \multirow{1}{*}{PALM}  & 2010.8 & 1000 & $\sim$ & 57.582 & 0.0672  \\ \cline{2-7}

\multirow{9}{*}{factors10NN \cite{10.1145/2049662.2049663}} & \multirow{1}{*}{APGnc$^+$} & 2013.2 & 1000 & $1.076*10^{-12}$ & 15.81  & 0.0674  \\  \cline{2-7}

\multirow{9}{*}{} & iPALM &  $1.82^*10^{6}$ 
& 1000 & $\beta$=0.99, $\alpha$=0.49 &  $1.431*10^{4}$&  29.52\\ \cline{2-7} 

\multirow{9}{*}{} & BPL & 2034.0 & 1000 & [0.10, 0.10, 0.10, 0.10]   & 15.975 & 0.0678 \\  \cline{2-7} 

\multirow{9}{*}{} & IBPG & 2043.0 & 1000 & \thead[l]{$\gamma_{t}=[0.10, ~ 0.10, ~ 0.10, ~ 0.10]$ \\ $\alpha_{t}=[0.067, ~ 0.067, ~ 0.067, ~ 0.067]$}   & 16.046& 0.1464 \\ \cline{2-7}

\multirow{9}{*}{} & \multirow{1}{*}{ABPL-cyclic} & 127.20 & 367.68 & 0.95 & 0.998  & $9.55*10^{-5}$  \\ \cline{2-7}

\multirow{9}{*}{} & \multirow{1}{*}{ABPL-random}   & 127.20 & 358.23 & 0.95 & 0.998 & $9.27*10^{-5}$ \\ \cline{2-7}

\multirow{9}{*}{} & \multirow{1}{*}{ABPL$^+$-cyclic}  & 127.25 & $\bm{138.62}$ & 0.99 & 0.991 & $9.59*10^{-5}$  \\ \cline{2-7}

\multirow{9}{*}{} & \multirow{1}{*}{ABPL$^+$-random} & $126.39$ & $\bm{166.79}$ & 0.98  & 0.993 & $7.47*10^{-5}$  \\ \hline

\end{tabular*}
}
\label{mattime}
\end{table*}

Based on Table \ref{mattime} and Figure \ref{figmatrix}, we observed that
(i) comparing with other algorithms, ABPL$^+$ outperforms other algorithms in both evaluations, which shows that ABPL$^+$ has better results in terms of iteration speed and fast convergence; 
(ii) The impact of adaptive momentum is superior to non-adaptive; 
(iii) The final extrapolated parameter of APGnc$^+$ is a very small value compared to other APALM algorithms, which further confirm 
% the Theorems \ref{algo_flaw}, \ref{cutoff},  Corollary \ref{PS0}, \ref{PS1} and 
our point; 
(iv) Moreover, we observed that the average performance of ABPL$^+$ might be further improved if the blocks of variables were randomly shuffled. 

\subsection{Sparse non-negative tensor decomposition (SNTD) with $\ell_0$-norm constrain}
\label{subsection: NTF}
Sparse non-negative tensor decomposition is a widely used method in machine learning, one of which is NCP (Non-negative Canonical Polyadic Decomposition) \cite{cong2015tensor,bro2003new,wang2021sparse,wang2021inexact}.

% Non-negative tensor decomposition is a widely used method in machine learning, one of which is NCP (Non-negative Canonical Polyadic Decomposition) \cite{cong2015tensor,bro2003new}. 
% Some works have been devoted to the tensor decomposition with sparse regularization, the works of \cite{kim2013sparse} and \cite{allen2012sparse} studied the sparse regularization for tensor decomposition using $\ell_1$-norm and trace norm, but they only focused on the unconstrained CP model without the non-negative constraint, and both \cite{xu2017globally} and  \cite{wang2021sparse} studied tensor decomposition problems with non-negativity constraints, but did not consider problems with non-convex constraints such as $\ell_0$-norm constraints, desipite the fact that adding $\ell_{0}$ norm is the most intuitive and effective way to satisfy sparsity.

Herein, we focus on a sparse non-negative tensor decomposition with $\ell_0$-norm constrain ($\ell_0$-SNTD). Given an tensor:  $\mathcal{X} \in \prod_{i=1}^{n} R^{d_{i}}$, it corresponds to the following mathematical model: 
\begin{align}
&& &\min\limits_{\left\{A_{i}\right\}_{i=1}^{n}}\frac{1}{2}||\mathcal{X}-{[[ \ \ {\left\{A_{i}\right\}_{i=1}^{n}}  
 \ ]] }||_{F}^{2},  \notag \\
&& &s. t. A_{i} \geq 0, ~ ||A_{i}||_{0} \leq s_{i} ({\forall} i \in \left[ N \right], ~ s_{i} \in \mathbb{N}).
\label{e43}
\end{align}
where $A_{i}\in\mathbb{R}^{d_{i}\times R}$, $[[ \ ]]$ represents Kruskal operator, $\odot$ represents the Khatri-Rao product.
%And as the subsection \ref{subsection: NMF} says, the constrained optimization problem with $\ell_0$-norm is NP-hard, and the extrapolated sequence generated by APGnc$^+$ algorithms is difficult to meet the sparsity constraint requirements, and most APALM algorithms have very strict restrictions on the extrapolated parameters and step size. And many, if not most, APALM algorithms only consider the non-negative limit rather than the $\ell_{0}$ norm \cite{le2020inertial,xu2017globally,wang2021sparse}. Thus, we will apply our and baseline algorithm on it. 
\begin{table*}
\renewcommand\arraystretch{1.4}
\centering
\caption{Comparison of algorithms applied on the PosterEmssion and on OngoingEEG: the results are the average of ten runs and the bold indicates the best numerical performance under the set termination conditions}
\scalebox{1}{
\begin{tabular*}{\linewidth}{p{2cm}|p{1.9cm}|p{1.7cm}|p{1.5cm}|p{4.2cm}|p{1.25cm}|p{2.6cm}}  
\toprule     
Data & Algorithm  &Objective function value & Time (seconds) &Final extrapolated parameters  & RelErr &$\Delta_{RelErr}$  (See Eq.~\ref{stop}) \\
\midrule   
\multirow{9}{*}{Ongoing} &  PALM  & $3.083*10^{5}$ & 250 & $\sim$ & 0.360 & $1.94*10^{-6}$  \\ \cline{2-7}

\multirow{9}{*}{-EEG \cite{wang2018increasing,wang2021inexact,wang2021sparse}} &  APGnc$^+$ & $3.063*10^{5}$ & 250 & $1.51*10^{-128}$ & 0.357  & $8.79*10^{-7}$  \\  \cline{2-7}

\multirow{9}{*}{} & iPALM &   $4.16*10^{5}$ &250 & $\beta$=0.99, $\alpha$=0.49 &  0.360 &  $1.71*10^{-6}$\\ \cline{2-7} 

\multirow{9}{*}{} & BPL  & $3.069*10^{5}$ & 250 & [0.10, 0.10, 0.10]   & 0.358 & $1.27*10^{-7}$ \\ \cline{2-7}
\multirow{9}{*}{} & IBPG  & $3.050*10^{5}$ & 250 & \thead[l]{$\gamma_{t}=[0.10, ~ 0.10, ~ 0.10]$ \\ $\alpha_{t}=[0.067, ~ 0.067, ~ 0.067]$}   & 0.356 & $5.49*10^{-7}$ \\ \cline{2-7}

\multirow{9}{*}{} & ABPL-cyclic & $3.042*10^{5}$ & 71.72 & 0.99 & 0.353  & $6.64*10^{-9}$  \\ \cline{2-7}

\multirow{9}{*}{} & ABPL-random & $3.042*10^{5}$ & 83.24 & 0.99 & 0.355& $7.31*10^{-9}$ \\ \cline{2-7}

\multirow{9}{*}{} & ABPL$^+$-cyclic & $3.021*10^{5}$ & $\bm{69.82}$ & 0.88 & 0.352 & $7.72*10^{-9}$  \\ \cline{2-7}

\multirow{9}{*}{} & ABPL$^+$-random & $3.031*10^{5}$ & $\bm{80.62}$ & 0.86  & 0.353 & $7.80*10^{-9}$ \\ \hline

\multirow{9}{*}{Posterior} & PALM  & $1.717*10^{4}$ & 600 & $\sim$ &  0.628 &  $1.23*10^{-5}$ \\ \cline{2-7}

\multirow{9}{*}{-Emission \cite{tichy_ondrej_2023_7646462}} & APGnc$^+$ & $1.694*10^{4}$ & 569.58 & $1.33*10^{-71}$ &0.620 & $4.00*10^{-6}$   \\  \cline{2-7}

\multirow{9}{*}{} & iPALM &  $1.963*10^{4}$ 
& 355.68 & $\beta$=0.99, $\alpha$=0.49 & 0.718 &  $1.560*10^{-6}$\\ \cline{2-7} 

\multirow{9}{*}{} & BPL & $1.682*10^{4}$ & 590.06 & [0.10, 0.10, 0.10, 0.10]   & 0.615 & $2.79*10^{-6}$\\ \cline{2-7}
\multirow{9}{*}{} & IBPG & $1.684*10^{4}$ & 455.62 & \thead[l]{$\gamma_{t}=[0.10, ~ 0.10, ~ 0.10, ~ 0.10]$ \\ $\alpha_{t}=[0.067, ~ 0.067, ~ 0.067, ~ 0.067]$}  & 0.616 & $2.247*10^{-6}$ \\ \cline{2-7}

\multirow{9}{*}{} & ABPL-cyclic & $1.674*10^{4}$ & 376.72 & 0.99 & 0.611  & $1.71*10^{-6}$  \\ \cline{2-7}

\multirow{9}{*}{} & ABPL-random   & $1.670*10^{4}$ & 342.91  & 0.99 & 0.611 & $7.87*10^{-7}$ \\ \cline{2-7}

\multirow{9}{*}{} & ABPL$^+$-cyclic & $1.675*10^{4}$ & $\bm{265.27}$ & 0.91 & 0.613  & $8.00*10^{-7}$  \\ \cline{2-7}

\multirow{9}{*}{} & ABPL$^+$-random & $1.670*10^{4}$ & $\bm{290.51}$ & 0.93  & 0.611 & $8.58*10^{-7}$  \\ \hline
\end{tabular*}
}
\label{tensortime}
\end{table*}

%%%%%%%%%%%%%%%%%%%%%%%%%%%%%%%%%%%%%%%%%%%%%%%%%%%%%%%%%%%%%%%%%%%%%%%%%%%%%%
\begin{figure}[hbpt]
    \centering \includegraphics[width=1\linewidth]{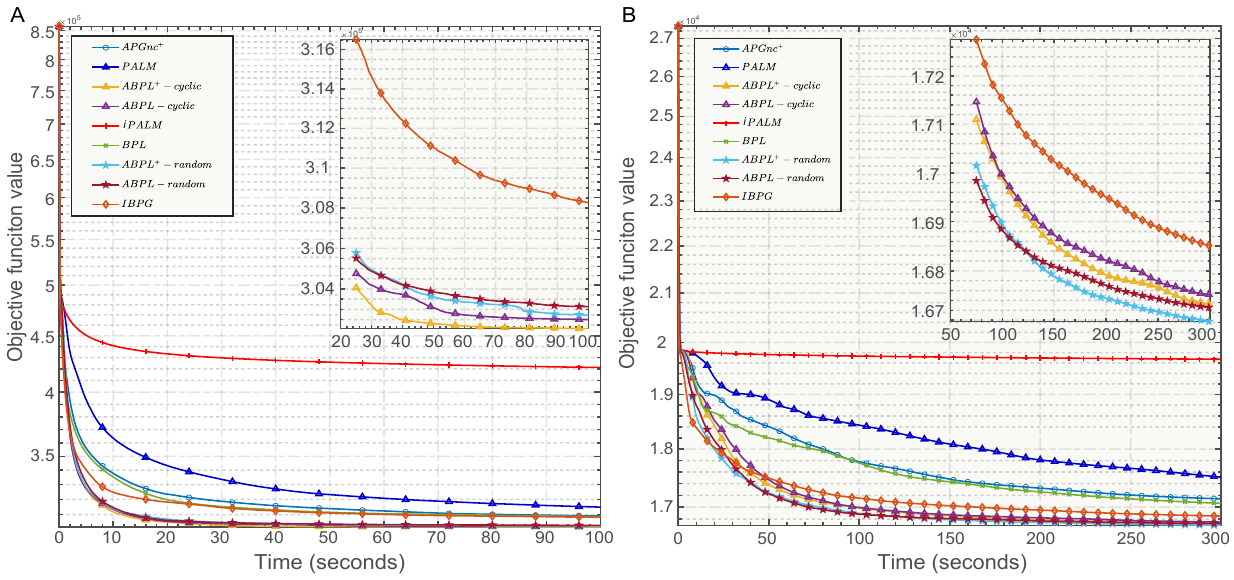}
    \caption{
    Comparison of the average convergence speed of the objective function of different algorithms on (A) the Ongoing-EEG dataset and (B) the PosteriorEmission dataset. 
    See Figure \ref{figmatrix} for other instructions.
    %The Convergence curves shows the average objective function value. A partial enlarged view of several curves with the fastest descending speed and the best effect is also be given, and the coordinate unit values and curve mark symbols of the partially enlarged figure are the same as those of the original figure. 
    }\label{figtensor}
\end{figure}
%%%%%%%%%%%%%%%%%%%%%%%%%%%%%%%%%%%%%%%%%%%%%%%%%%%%%%%%%%%%%%%%%%%%%%%%%%%%%%
\subsubsection{Solving $\ell_0$-SNTD using ABPL$^+$}
\myspace{0}Similarly, we analyze the relevant properties of the sparse non-negative tensor decomposition problem. First, compare with Eq. (\ref{e11}), then:  
\begin{align}
    && H(\left\{{A_{i}} \right\}^{n}_{i=1})&=\frac{1}{2}||\mathcal{X}-{[[ \ \ {\left\{A_{i}\right\}_{i=1}^{n}} \ ]] }||_{F}^{2},~ F_{i}(A_{i})&=\delta(A_{i}). \notag 
\end{align}
the definition of $\delta(A_{i})$ is the same as Eq. (\ref{delta}). 
\myspace{0}Let $X^{(j)} \in  R^{d_{j} \times \prod_{i=1, i \neq j}^{n} d_{i}}$ represent the mode-n unfolding of original tensor $\mathcal{X}$, the mode-n unfolding of the estimated tensor in Kruskal operator [[$ \ \prod_{i=1}^{n} A_{i}  \ $]] can be written as $A_{i}B_{i}^{T}$, where: 
\begin{center}
     $B_{i}= (A_{n}\odot. . . A_{i+1}\odot A_{i-1}. . . \odot A_{1})$.
    \end{center}
in BCD frame \cite{wang2021sparse}, Eq. (\ref{e43}) can be decomposed into n sub-problems: 
\begin{center}
     $\min\limits_{A_{i}}\mathcal{F}(A_{i})=\frac{1}{2}||X^{(i)}-A_{i}(B_{i})^{T}||_{F}^{2}$,
     \\
$ s. t. A_{i} \geq 0, ~ ||A_{i}||_{0} \leq s_{i} ({\forall} i \in \mathbb{N})$.
\end{center}
the partial gradient (or partial derivative) of $\mathcal{F}(A_{i}) $ with respect to $A_{i}$ is
\begin{align}
&& &\nabla_{A_{i}} H(\left\{{A_{i}} \right\}^{n}_{i=1})=\frac{\partial}{\partial A_{i}}\mathcal{F}(A_{i})=A_{i}(B_{i})^{T}(B_{i})-X^{(i)}B_{i}. \notag
 \end{align}
By Lemma \ref{test1}, $H$ are lipschitz gradient constant is $L(A_{i})=(B_{i})^{T}*B_{i}$, and $c_{i}=\gamma_{i}*L(A_{i}), ~ \gamma_{i}>1$.

Therefore, it is easy to verify that sparse non-negative tensor decomposition with $\ell_0$-norm constrain satisfies Assumption \ref{assump1}, and the proximal operator corresponding to $\ell_0$-SNTD is: 
\begin{align}
 && A_{i}^{k+1} &\in prox_{\frac{1}{c^{k}_{i}} R_{i}(A^{k}_{i})}
 (U^{k}_{i}) \notag \\
 && &=\operatorname*{\arg\min} \limits_{A}\left\{\frac{c_{i}^{k}}{2}\left|\right|A-U^{k}_{i}\left|\right|_{F}^{2}+F_{i}(A^{k}_{i}) \right\}  \notag \\ 
 && &= \operatorname*{\arg\min} \limits_{A}\left\{\frac{c_{i}^{k}}{2}\left|\right|A-U^{k}_{i}\left|\right|_{F}^{2}: A\geq 0, \left|\left|A\right|\right|_{0}\le s\right\}.  \notag
\end{align}
where: 
\begin{center}
$U^{k}_{i}=A^{k}_{i}-\frac{1}{c^{k}_{i}} \nabla_{X_{i}} H(\left\{{A^{k+1}_{j}} \right\}^{i-1}_{j=1}, A^{k}_{i}, \left\{{A^{k}_{i}} \right\}^{n}_{j=i+1})$. 
\end{center}

\subsubsection{Numerical tests}
We test the proposed and baline algorithms on the the datasets of OngoingEEG\footnote{\url{https://github.com/wangdeqing/Ongoing_EEG_Tensor_Decomposition}} \cite{wang2018increasing,wang2021inexact} and PosteriorEmission\footnote{\url{https://zenodo.org/record/7646462}} \cite{tichy_ondrej_2023_7646462}, the dataset of OngoingEEG is 3-dimension tensor($54\times146\times510$), the dataset of PosteriorEmission is 4-dimension tensor($240\times40\times12\times8$), we set non-zero parameter is $30\%$ of the nonzero elements for each matrix. For OngoingEEG, the initial number of components $R$ was set by 40. And for PosteriorEmission, the initial number of components $R$ was set by 50. The initial point is set to a uniformly distributed random sparse positive definite matrix. Set each initial hyperparameter according to the initial conditions of the algorithms as: $t=1.3, ~ \beta_{1}=0.2, ~ \beta_{max}=0.9999, ~ \gamma_{i}=1.5$. 

In order to show the effectiveness of the proposed algorithm, we will also evaluate them from two aspects: (1) we run all algorithms ten times with the same hyperparameters, each time using the different random initialization point, and let all algorithms run for the same length of time. For OngoingEGG, the running time was set by 100 seconds, for PosteriorEmission, the running time was set by 300 seconds. We plotted the curves of the average objective function values of OngoingEEG and PosteriorEmission as a function of time under different algorithms in Figure \ref{figtensor}(A) and Figure \ref{figtensor}(B) respectively. (2) For the second aspect, define $RelErr^{k}=\frac{||\mathcal{X}-{[[ \ \prod_{i=1}^{n} A^{k}_{i}  
 \ ]] }||_{F}}{||\mathcal{X}||_{F}}$, using the same stopping criteria with Eq. (\ref{stop}) 
% \begin{align}
% RelErr^{k}=\frac{||\mathcal{X}-{[[ \ \prod_{i=1}^{n} A^{k}_{i}  
%  \ ]] }||_{F}}{||\mathcal{X}||_{F}}. \notag \\  
% \Delta_{RelErr^{k}}=|RelErr^{k+1}-RelErr^{k}|<\epsilon.
% \label{stop1}
% \end{align}
, and for OngoingEEG, we set $\epsilon=1e-8$ and maximum running time was set by 250 seconds. For PosteriorEmission, we set $\epsilon=1e-6$ and maximum running time was set by 600 seconds.
% \begin{center}
% $RelErr=\frac{| \ ||\mathcal{X}-{[[ \ \prod_{i=1}^{n} A^{k+1}_{i}  
%  \ ]] }||_{F}-||\mathcal{X}-{[[ \ \prod_{i=1}^{n} A^{k}_{i}  
%  \ ]] }||_{F} |}{||\mathcal{X}||_{F}}<\epsilon$
% \end{center}

Then in Table \ref{tensortime}, we run all algorithms ten times with the same hyperparameters and use the different random initialization point each time, Table \ref{tensortime} shows the average time and the values of various indicators when the running result of the algorithms reach the termination condition. 

Similarly, just as the Table \ref{tensortime} and Figure \ref{figtensor} demonstrates, we observe that 
(\romannumeral1) comparing with other algorithms, ABPL$^+$ outperforms other algorithms in both evaluations, which shows that ABPL$^+$ has better results in terms of iteration speed and fast convergence;  
(\romannumeral2) furthermore, the impact of adaptive momentum is superior to non-adaptive;  
(\romannumeral3) And the final extrapolated parameter of APGnc$^+$ is a very small value compared to other APALM algorithms, which further confirm 
 % the Theorems \ref{algo_flaw}, \ref{cutoff}, Corollary \ref{PS0}, \ref{PS1} and 
our point; 
(\romannumeral4) Moreover, we observed that the average performance of ABPL$^+$ might be further improved if the blocks of variables were randomly shuffled. 

\section{Conclusion}
In this paper, we propose an accelerated block proximal linear framework with adaptive mometum (ABPL$^+$) for solving non-convex non-smooth problems. Under our analytical method, we discuss the reasons for why the extrapolation phase could fail in specific algorithmic frameworks and show the effectiveness of our algorithm, which can not only solve the problem of extrapolation failure and extend to multivariate, but also resolves the issue that commonly used accelerated proximal alternation algorithms need to impose stringent limitations on the selection of extrapolation parameters and step size. 
Furthermore, given some modest assumptions, our method demonstrated that our algorithm is a  monotonically decreasing method in terms of objective function values and demonstrated the effectiveness of our algorithm's random update order, we also succinctly and intuitively proving that the entire sequence converges to the critical point. Further assuming the Kurdyka–Łojasiewicz property of the objective function, we estimate our algorithm's asymptotic convergence rate, which establish the global convergence of the sequences produced by our algorithm. Numerical experimental results of $\ell_0$-msNMF and $\ell_0$-NSTD are very encouraging.

\begin{appendices}

% \section{APGnc algorithm framework}
% \label{APGncAppdix}
% The APGnc algorithm is given in Algorithm \ref{APGnc}: 

% \begin{algorithm}[!htbp]
%     \caption{ \small APGnc: APG non-convex problem \cite{yao2016more,li2017convergence}}\label{APGnc}
%   \begin{small}
%     \SetAlgoLined
%      \LinesNumbered
%         \KwIn{$z^{1} =x^{1}  \in \mathrm{dom} \ J$, $k_{max}=c, ~t>1,         ~\beta_{max} \in [0,1),~\beta_{1} \in [0,\beta_{max}]$.}
%         \KwOut{$x^{k+1}$.}
%         \For{$k = 1$ to $k_{max}$}
%         {
% 	  $\sigma^{k} \in (0, \frac{1}{L_{\nabla H_{x^{k}}}})$, $x^{k+1} \in prox_{\sigma^{k} F} (z^{k}-\sigma^{k} \nabla_x h(z^{k})) $,  
%         $y^{k+1}$=$x^{k+1}$+$\beta_{k} (x^{k+1}-x^{k})$,
        
%         % \If{ \begin{equation}\label{ifAPGnc}
%         \eIf{\begin{equation}\label{ifAPGnc}J(y^{k+1}) \leq J(x^{k})\end{equation}}
%         % \end{equation}} 
%         {
%          $z^{k+1}=y^{k+1}$.
%         }{
%          $z^{k+1}=x^{k+1}$.     
%         }      
%         }
%   \end{small}
% \end{algorithm}

\section{Potential reasons for the failure of the extrapolation step}
\label{Appdix2}
Some fundamental theorems need to be proven before delving deeper into algorithmic defects. These theorems are general and do not specific to the algorithms discussed in this article. We will use them to explain why extrapolation methods in certain iterative optimization algorithms fail. 
% More specifically, for certain scenarios, the extrapolated sequence does not necessarily belong to the set $\mathrm{dom} ~J=S$ in the algorithm framework we are considering. We will demonstrate this scenario using measure theory, and it is highly likely to occur in specific problem instances. 
\begin{theorem}
$\mu$ is a Lebesgue measure. From Definition 
\ref{lebesgue}, \ref{truncated}, define $\Gamma_{\left[ n \right]}$ as the set system composed of all subsets of the set $\left[ n \right]$, and $K_{j}={\left\{k_{j(i)}\right\}}_{j(i)=1}^{m,m \leq n}, ~ \cup_{j=1}^{2^{n}}K_{j}=\Gamma_{\left[ n \right]}$, $\left\{x_{i}\right\}_{i\in \left[ n \right]} \in A $, then $\mu(\cup_{j=1}^{2^{n}} A|_{K_{j}})=0$: 
\label{cutoff}
\end{theorem}
\begin{proof}
From the Lemma \ref{cmeasure} and single point set measure is 0, thus we have: 
\begin{align}
   && &\mu(A|_{K_{j}})=\mu(\left\{(x_{\left[ n \right] \setminus \left\{k_{i}\right\}_{i=1}^{m,m \leq n} })\right\}) \mu(x_{\left\{k_{i}\right\}_{i=1}^{m,m \leq n}})=0. \notag
\end{align}
according to the Definition \ref{dmeasure}: 
\begin{center}
  $\mu(\cup_{j=1}^{2^{n}} A|_{K_{j}})\leq \sum_{j=1}^{2^{n}} \mu(A|_{K_{j}})=0$.  
\end{center}
which means that Theorem \ref{cutoff} holds. 
\end{proof}
\myspace{0}Therefore we have the following corollary: 
\begin{corollary}
Define X as the complete set, let $S=\mathrm{dom}~J$, if S is a set composed of truncated sets (Definition \ref{truncated}) of the complete set X, then P(S)=0 (P is a probability measure)    
\label{PS0}
\end{corollary}
\begin{proof}
Define $P(X)=1$, from Theorem \ref{cutoff} and Definition \ref{pro}, it is obviously true.
\end{proof}
\myspace{0} A most intuitive example is the case of multi-block problems with $\ell_0$-norm constraints. 
\begin{corollary}
If we need $||y^{k}_{i}||_{0} \leq s_{i}~({\forall} i \in \left[ n \right])$, then $P(y^{k} \in S)=0$ ($S=\mathrm{dom}~ J$, P is a probability measure)
 \label{PS1}
 \end{corollary}
\begin{proof}
This is because: 
\begin{center}
$S=\mathbb{R}^{n}|_{{\left\{k_{i}\right\}}_{i=1}^{m, m \leq n}}=\left\{(x_{\left[ n \right]  }): (x_{{\left\{k_{i}\right\}}_{i=1}^{m, m \leq n}} )= 0 \in \mathbb{R}^{n} \right\}$.
\end{center}
which meets Theorem \ref{cutoff} and Corollary \ref{PS0}. 
\end{proof}
Therefore, the extrapolation step in APGnc and APGnc$^+$ (\cite{yao2016more,li2017convergence}) is almost in failure state.
 \end{appendices}
 
\section*{Acknowledgment}
The work of W. Y and W. M. was supported in part by the National Natural Science Foundation of China (62262069), in part by the Program of Yunnan Key Laboratory of Intelligent Systems and Computing (202205AG070003), in part by the Yunnan Fundamental Research Projects under Grant (202201AT070469, 202301BF070001-019). 
\balance
\small
\bibliographystyle{IEEEtran}
\bibliography{refer}
\end{document}